\documentclass{amsart}

\usepackage[utf8]{inputenc} 
\usepackage[T1]{fontenc} 
\usepackage[english]{babel}
\usepackage{amsmath}
\usepackage{amssymb}
\usepackage{graphicx}

\theoremstyle{plain} 
\newtheorem{theorem}{Theorem} 
\newtheorem{lemma}[theorem]{Lemma} 
 
\newtheorem{definition}[theorem]{Definition} 
\newtheorem{example}{Example} 
\theoremstyle{remark} 
\newtheorem*{remark}{Remark} 

\newtheorem*{conjecture}{Conjecture}

\DeclareMathOperator{\capp}{cap}
\DeclareMathOperator{\supp}{supp}
\DeclareMathOperator{\inte}{int}
\DeclareMathOperator{\exte}{ext}
\DeclareMathOperator{\mre}{Re}
\DeclareMathOperator{\mim}{Im}
\DeclareMathOperator{\Tr}{Tr}
\DeclareMathOperator{\Res}{Res}
\DeclareMathOperator{\clos}{clos}
\DeclareMathOperator{\ind}{ind}
\DeclareMathOperator{\ac}{ac}
\DeclareMathOperator{\ess}{ess}
\DeclareMathOperator{\codim}{codim}
\DeclareMathOperator{\sgn}{sgn}
\begin{document} 

\title[Plasmonic eigenvalue problem for corners]{Plasmonic eigenvalue problem for corners: limiting absorption principle and absolute continuity in the essential spectrum}
\date{\today} 

\author{Karl-Mikael Perfekt}
\address{Department of Mathematics and Statistics,
	University of Reading, Reading RG6 6AX, United Kingdom}
\email{k.perfekt@reading.ac.uk}

\begin{abstract}
We consider the plasmonic eigenvalue problem for a general 2D domain with a curvilinear corner, studying the spectral theory of the Neumann--Poincar\'e operator of the boundary. A limiting absorption principle is proved, valid when the spectral parameter approaches the essential spectrum. Putting the principle into use, it is proved that the corner produces absolutely continuous spectrum of multiplicity 1. The embedded eigenvalues are discrete. In particular, there is no singular continuous spectrum.

\end{abstract}

\subjclass[2010]{}
\maketitle
\section{Introduction}
{
Let $\Gamma$ be a piecewise $C^3$ Jordan curve in $\mathbb{R}^2 \simeq \mathbb{C}$ with a single curvilinear corner of angle $\alpha \neq \pi$, $0 < \alpha < 2\pi$, see Figure~\ref{fig:Gamma}. We refer to the interior and exterior of $\Gamma$ as $\Gamma_{\inte}$ and $\Gamma_{\exte}$. For $\epsilon_r \in \mathbb{C}$, $\epsilon_r \notin \{0,1\}$, a relative permittivity which we also understand as the spectral parameter, the plasmonic eigenvalue problem seeks a potential $U \colon \Gamma_{\inte} \cup \Gamma_{\exte} \to \mathbb{C}$ such that $U(x) = o(1)$, $x \to \infty$, and 
\begin{equation} \label{eq:plasmonicproblem}
\begin{cases}
\Delta U(x) = 0, \quad &x \in \Gamma_{\inte} \cup \Gamma_{\exte}, \\
\Tr_{\inte} U(x) = \Tr_{\exte} U(x), \quad &x \in \Gamma, \\
\partial_{n}^{\exte} U(x) = \epsilon_r \partial_{n}^{\inte} U(x), \quad &x \in \Gamma.
\end{cases}
\end{equation}
Here $\Tr_{\inte} U$ and $\Tr_{\exte} U$ refer to interior and exterior traces, the limiting boundary values of $U$ on $\Gamma$ from the inside and outside of $\Gamma$, respectively. Similarly, $\partial_{n}^{\inte} U$ and $\partial_{n}^{\exte} U$ denote outward unit normal derivatives on the boundary, calculated with interior and exterior approach. In general, these are understood in a distributional sense based on Green's formula.

\begin{figure}
	\centering
	\includegraphics[width =0.8\linewidth]{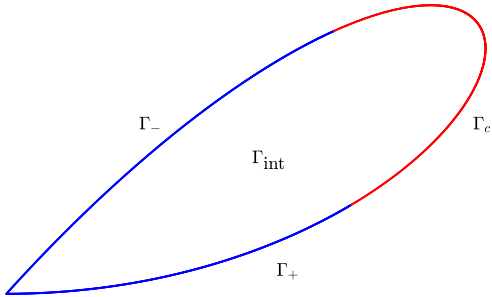}
	\caption{A piecewise $C^3$-curve $\Gamma$ with a corner.}
	\label{fig:Gamma}
\end{figure}

To solve the plasmonic problem, we make an ansatz with a double layer potential, see Example~\ref{ex:polar}, which leads to the eigenvalue problem 
$$(K - \lambda) \rho = 0, \quad \lambda = \frac{\epsilon_r + 1}{\epsilon_r - 1}.$$
Here $K$ denotes the Neumann--Poincar\'e (NP) operator, or the direct value of the double layer potential,
$$K\rho(x) = \frac{1}{\pi} \int_{\Gamma} \frac{\langle y - x, n_y \rangle}{|x-y|^2} \, \rho(y) \, d\sigma(y), \quad x \in \Gamma,$$
where $n_y$ denotes the unit outward normal at $y \in \Gamma$, and $\sigma$ is the arc length measure along $\Gamma$. To understand the plasmonic eigenvalue problem we will thus investigate the NP operator $K$. We refer the reader to \cite{AMRZ17, AK16} for further study and illustration of the connection between the spectrum of the NP operator and plasmon resonances.

Because $\Gamma$ has a corner, the spectral theory of $K$ has a number of rather special features. It is known \cite{Mit02, Shele90}, that there is a Jordan curve $\Sigma_{1,\alpha}$ with interior $\widetilde{\Sigma}_{1,\alpha}$, illustrated in Figure~\ref{fig:Sigma}, such that every $\lambda \in -\widetilde{\Sigma}_{1,\alpha} \cup \widetilde{\Sigma}_{1,\alpha}$ is an eigenvalue of $K \colon L^2(\Gamma) \to L^2(\Gamma)$. On the other hand, if we consider $K$ as an operator on the Sobolev space $H^{1/2}(\Gamma)$ we shall prove the following theorem. Note that eigenvectors to eigenvalues $\lambda \neq 1$ of $K \colon H^{1/2}(\Gamma) \to H^{1/2}(\Gamma)$ correspond to solutions of \eqref{eq:plasmonicproblem} with $\int_{\mathbb{R}^2} | \nabla U |^2 \, dx < \infty$, see Example~\ref{ex:polar}.
 \setcounter{theorem}{0}
\renewcommand{\thetheorem}{\Alph{theorem}}
\begin{theorem} \label{thm:main1}
The point spectrum of $K \colon H^{1/2}(\Gamma) \to H^{1/2}(\Gamma)$ consists of a simple eigenvalue at $\lambda = 1$ and a (possibly finite) sequence $\{\lambda_n\} \subset (-1,1)$ that is symmetric around $\lambda = 0$ and which can only accumulate at $0$ and $\pm |1-\alpha/\pi|$,
$$\sigma_{p}(K, H^{1/2}(\Gamma)) = \{\lambda_n\}\cup \{1\}.$$
\end{theorem}
The action of $K$ on $H^{1/2}(\Gamma)$ is distinguished, in that $H^{1/2}(\Gamma)$ may be given an equivalent special norm $\| \cdot  \|_{\mathcal{E'}}$ in which $K \colon H^{1/2}(\Gamma) \to H^{1/2}(\Gamma)$ is self-adjoint. By dilation we may assume that the logarithmic capacity of $\Gamma$ satisfies $\capp(\Gamma) < 1$. Then any $f \in H^{1/2}(\Gamma)$ may be written $f = Sg$ for a unique $g \in H^{-1/2}(\Gamma)$, where $S$ is the direct value of the single layer potential on $\Gamma$, and then
$$\|f\|_{\mathcal{E'}}^2 = \langle g, f \rangle_{L^2(\Gamma)}.$$
Here we have chosen the sign of $S$ so that this is a positive quantity, see Section~\ref{subsec:singlelayer}.
When speaking of $K \colon H^{1/2}(\Gamma) \to H^{1/2}(\Gamma)$ as a self-adjoint operator, we always consider $H^{1/2}(\Gamma)$ to be equipped with the norm $\| \cdot  \|_{\mathcal{E'}}$.

The essential spectrum of $K \colon H^{1/2}(\Gamma) \to H^{1/2}(\Gamma)$ has been computed in \cite{BZ19, PP17}, with two somewhat different approaches,
$$\sigma_{\ess}(K, H^{1/2}(\Gamma)) = [-|1-\alpha/\pi|, |1-\alpha/\pi|].$$
Both contributions draw on variational formulations of the plasmonic problem \eqref{eq:plasmonicproblem}. It is in fact rather difficult to give a useful description of the $H^{1/2}(\Gamma)$-norm on $\Gamma$, see Section~\ref{sec:prelim}. In \cite{LS19}, the authors constructed curves $\Gamma$ for which some of the eigenvalues $\{\lambda_n\}$ are embedded in $[-|1-\alpha/\pi|, |1-\alpha/\pi|]$. See also \cite{HKL17} for numerical examples.

\begin{figure}
	\centering
	\includegraphics[width =0.8\linewidth]{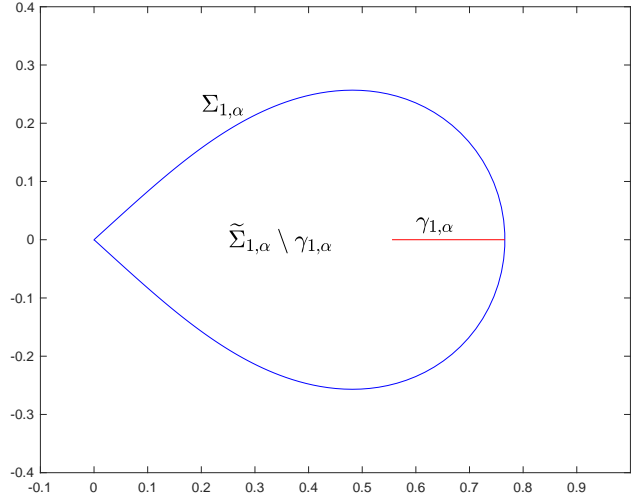}
	\caption{The sets $\Sigma_{1, \alpha}$, $\gamma_{1,\alpha}$, and $\widetilde{\Sigma}_{1,\alpha} \setminus \gamma_{1,\alpha}$ illustrated for $\alpha = 4\pi/9$. }
	\label{fig:Sigma}
\end{figure}

The second main theorem is a limiting absorption principle which holds when the spectral parameter $\lambda$ approaches the essential spectrum. For simplicity, let us describe the result in the introduction only when $0 < \alpha < \pi$ and $\mre \lambda > 0$. The case when $\alpha > \pi$ requires only a few sign-changes, and the spectral properties of $K$ and $-K$ on $H^{1/2}(\Gamma)$ are nearly identical due to a symmetry. 

Let $\gamma_{1, \alpha} = [1-\alpha/\pi, \sin((\pi-\alpha)/2)]$. Then $(0, 1-\alpha/\pi)$ is contained in the region $\widetilde{\Sigma}_{1,\alpha} \setminus \gamma_{1, \alpha}$. We will consider a certain biholomorphic map
\begin{equation} \label{eq:intromu}
\mu_{\alpha} \colon \widetilde{\Sigma}_{1,\alpha} \setminus \gamma_{1, \alpha} \to \{z \in \mathbb{C} \, : \, |\mre z| < 1/2, \; \mim z < 0\}
\end{equation}
such that $\mu_{\alpha}((0,1-\alpha/\pi)) = i\mathbb{R}_-$, $\mu_\alpha(\bar{\lambda}) = -\overline{\mu_{\alpha}(\lambda)}$, and $\mre \mu_{\alpha}(\lambda) < 0$ for $\mim \lambda > 0$. To each $\lambda \in \widetilde{\Sigma}_{1,\alpha} \setminus \gamma_{1, \alpha}$ we attach a singular function $q_\lambda \in L^2(\Gamma)$ such that
$$q_\lambda(x) = \phi(s)s^{-\mu_{\alpha}(\lambda)},$$
where $s$ denotes the arc length distance of $x$ to the corner and $\phi$ is a cut-off function around $s = 0$. Note that when $\mim \lambda > 0$, the operator $K-\lambda \colon H^{1/2}(\Gamma) \to H^{1/2}(\Gamma)$ is invertible by self-adjointness.
\begin{theorem} \label{thm:main2}
	For $g \in H^1(\Gamma)$ and $\lambda \in \widetilde{\Sigma}_{1,\alpha}$ with $\mim \lambda > 0$, let 
	$$f_\lambda = (K-\lambda)^{-1} g \in H^{1/2}(\Gamma).$$
	Then there are unique $c_\lambda \in \mathbb{C}$ and $\tilde{f}_\lambda \in H^1(\Gamma)$ such that
	$$f_\lambda = c_\lambda q_\lambda + \tilde{f}_\lambda.$$
	Let $H_\lambda \colon H^1(\Gamma) \to \mathbb{C} \oplus H^1(\Gamma)$ be the corresponding map 
	$$H_\lambda g = (c_\lambda, \tilde{f}_\lambda).$$
	Then the operator-valued function $\lambda \to H_\lambda$ has a meromorphic extension to $\widetilde{\Sigma}_{1,\alpha} \setminus \gamma_{1, \alpha}$ with simple poles at $\{\lambda_n\} \cap (0, 1-\alpha/\pi)$, and no other poles in $(0, 1-\alpha/\pi)$. 
	
	In particular, if $u \in (0, 1-\alpha/\pi)$ and $u$ is not an eigenvalue of $K \colon H^{1/2}(\Gamma) \to H^{1/2}(\Gamma)$, then the limit
	$$h_u = \lim_{\varepsilon \to 0^+} (K-(u+i\varepsilon))^{-1} g$$
	exists in $L^2(\Gamma)$ and is of the form $h_u = c_u q_u + \tilde{h}_u$, where $c_u \in \mathbb{C}$ and $\tilde{h}_u \in H^1(\Gamma)$. Furthermore, $h_u$ is the unique solution of this form of the equation
	$$(K-u)h_u = g.$$
\end{theorem}
Thus, by a Sobolev embedding theorem, the limiting solution has the asymptotics
$$h_u(x) = d_u + c_us^{-\mu_{\alpha}(u)} + O(\sqrt{s}),$$
for $x$ close to the corner, for some $c_u, d_u \in \mathbb{C}$. Recall here that $\mre \mu_{\alpha}(u) = 0$ for $u \in (0, 1-\alpha/\pi)$. One could of course have approached the essential spectrum in $\mim \lambda < 0$ instead. One would then obtain a similar statement with singular functions of the form $\tilde{q}_\lambda(x) = \phi(s) s^{\mu_\alpha(\lambda)}$. 

Many authors prefer to work with the $L^2$-adjoint $K^\ast$, rather than $K$. Theorem~\ref{thm:main2} also implies a limiting absorption principle for $K^\ast \colon H^{-1/2}(\Gamma) \to H^{-1/2}(\Gamma)$, see Section~\ref{sec:remarks}. Numerical illustrations of this principle can be found in \cite[Figures~7 and 8]{HR19}.

Theorem~\ref{thm:main2} can be used to study the spectral theory of $K \colon H^{1/2}(\Gamma) \to H^{1/2}(\Gamma)$. In \cite{HLY17}, the spectral resolution of $K \colon H^{1/2}(\Gamma) \to H^{1/2}(\Gamma)$ is given explicitly in the case that $\Gamma$ is a lens (which has two corners of the same angle). To the author's knowledge, this is the only example where the finer spectral details have been known previously. A spectral theory problem which is similar in spirit to the plasmonic problem and which can be diagonalized explicitly can also be extracted from the second model problem described in \cite{BCC13}. Formulas for spectral measures have also been provided in certain periodic settings \cite{CO01, Milt01}.

Perhaps most notable is Carleman's 1916 thesis \cite{Carl16}, where the resolvent of $K$ is investigated in great detail. Carleman's study was significantly ahead of its time, as the modern tools of functional analysis, function spaces, and spectral theory were not available. At the very end of his thesis, he nonetheless writes down a resolution of $K^\ast$. Using Theorem~\ref{thm:main2} we can put Carleman's spectral resolution on formal footing.

Let $K = \int_{-1}^1 \lambda \, dE(\lambda)$ be the spectral decomposition of $K \colon H^{1/2}(\Gamma) \to H^{1/2}(\Gamma)$. For $g,h \in H^{1/2}(\Gamma)$, we denote the associated spectral measure by $\nu_{g,h}$,
$$\nu_{g,h}(\Lambda) = \langle E(\Lambda) g, h \rangle_{\mathcal{E}'}, \quad \Lambda \subset [-1,1] \textrm{ Borel}.$$

\begin{theorem} \label{thm:main3}
	In addition to the discrete set of eigenvalues $\{\lambda_n'\} = \{\lambda_n\} \cup \{1\}$ described in Theorem~\ref{thm:main1}, the spectrum of $K \colon H^{1/2}(\Gamma) \to H^{1/2}(\Gamma)$ consists of an absolutely continuous part of multiplicity 1,
	$$\sigma_{\ac}(K) = [-|1-\alpha/\pi|, |1-\alpha/\pi|].$$
	In particular, there is no singular continuous spectrum.
	
	Furthermore, if $g,h \in H^1(\Gamma)$, then the spectral measure 
	$$\nu_{g,h} = \tau \, d\lambda + \sum_{n} \tau_n \delta_{\lambda_n'}$$
	has a continuous density $\tau$ which is real analytic in $(-|1-\alpha/\pi|, 0)$ and $(0, |1-\alpha/\pi|)$.
\end{theorem}
In fact more will be proven in Theorem~\ref{thm:spectrum2}; a semi-explicit spectral resolution of $K \colon H^{1/2}(\Gamma) \to H^{1/2}(\Gamma)$ in terms of eigenfunctions of $K \colon L^2(\Gamma) \to L^2(\Gamma)$ will be provided. Our analysis also extends to piecewise $C^3$-domains with multiple corners of various angles, but discussion will be restricted to the case of a single corner, primarily for notational reasons.

To finish the introduction we apply Theorem~\ref{thm:main3} to the polarizability problem.
\begin{example}[Polarizability] \label{ex:polar} \rm
Viewing $\Gamma_{\inte}$ as an inclusion in the infinite space $\mathbb{R}^2$, let $\epsilon_r \in \mathbb{C}$, $\epsilon_r \notin \{0,1\}$, denote its relative permittivity. To find the polarizability tensor of $\Gamma_{\inte}$ we solve, for a field $e \in \mathbb{R}^2$, the problem
\begin{equation*}
\begin{cases}
\int_{\mathbb{R}^2} |\nabla U - e|^2 \, dx < \infty, \\
\Delta U(x) = 0, \quad &x \in \Gamma_{\inte} \cup \Gamma_{\exte}, \\
\Tr_{\inte} U(x) = \Tr_{\exte} U(x), \quad &x \in \Gamma, \\
\partial_{n}^{\exte} U(x) = \epsilon_r \partial_{n}^{\inte} U(x), \quad &x \in \Gamma.
\end{cases}
\end{equation*}
This problem has a solution if and only if there is a solution of the form
\begin{equation*}
U(x) = \begin{cases} \epsilon_r^{-1}(\langle e, x \rangle + D\rho(x)), \quad &x \in \Gamma_{\inte}, \\ \langle e, x \rangle + D\rho(x), \quad &x \in \Gamma_{\exte},\end{cases}
\end{equation*}
where $D\rho$ is the double layer potential of some $\rho \in H^{1/2}(\Gamma)$. See for example \cite{PP14}. The condition $\Tr_{\inte} U = \Tr_{\exte} U$ yields, by the jump formulas 
$$\Tr_{\inte} Df = (K+I)f, \quad \Tr_{\exte} Df = (K-I)f,$$
that
$$a + (K+I)\rho = \epsilon_r(a + (K-I)\rho),$$
where $a(x) = \langle e, x \rangle$, $x \in \Gamma$. That is,
$$(K + z) \rho = -a, \quad z = -\frac{\epsilon_r + 1}{\epsilon_r - 1}.$$
The area-scaled polarizability tensor $\omega(z) \colon \mathbb{R}^2 \to \mathbb{R}^2$ is given by
$$\omega(z)e = \frac{\epsilon_r - 1}{|\Gamma_{\inte}|} \int_{\Gamma_{\inte}} \nabla U(x) \, dx.$$
It is a linear map determined, for $j,k=1,2$, by 
$$\omega_{jk}(z) := \langle \omega(z)e_j, e_k \rangle = \frac{2}{|\Gamma_{\inte}|} \int_{\Gamma} \rho_j b_k \, d\sigma =  -\frac{2}{|\Gamma_{\inte}|} \langle (K+z)^{-1} a_j, Sb_k \rangle_{\mathcal{E}'},$$
where $e_1,e_2$ is the standard basis of $\mathbb{R}^2$, and $a_j(x) = \langle e_j, x \rangle$, $b_j(x) = \langle e_j, n_x \rangle$, and $\rho_j = -(K+z)^{-1}a_j$, $j=1,2$. The polarizability tensor $\omega(z)$ is thus clearly well-defined when $-z \notin \sigma(K, H^{1/2}(\Gamma))$, and in particular when $\mim z > 0$. For $z \in \mathbb{R}$ we define the limit polarizabilities as
$$\omega_{jk}^+(z) = \lim_{\varepsilon \to 0^+} \omega_{jk}(z + i\varepsilon), \quad j,k = 1,2,$$
whenever this is well-defined. 

Let $I = (-|1-\alpha/\pi|, |1-\alpha/\pi|)$. Since $a_j, Sb_k \in H^1(\Gamma)$, there is by Theorem~\ref{thm:main3} a density $\tau_{jk} \in L^1(I)$ which is real analytic on $I\setminus \{0\}$, a sequence of eigenvalues $\{\lambda_n'\}$ that can only accumulate at $0$ and the endpoints of $I$, and a sequence $(\tau_{jk}(n)) \in \ell^1$ such that
$$-\omega_{jk}(z) = \int_I \frac{\tau_{jk}(u)}{u + z} \, du + \sum_n \frac{\tau_{jk}(n)}{\lambda_n' + z}, \quad \mim z > 0.$$
By \cite[Theorem~5.2]{HP13}, $\tau_{jj}(u) \geq 0$ and $\tau_{jj}(n) \geq 0$, for $j=1,2$, $u \in I$, and all $n$. 

Furthermore, the limit polarizability $\omega_{jk}^{+}(z)$ is real analytic for $z \in \mathbb{R}$, except possibly when $-z \in \{\lambda_n'\}\cup\{0, \pm |1-\alpha/\pi|\}.$ Figure~\ref{fig:limpol} gives a numerical illustration of the limit polarizability $\omega_{11}^+$ in the case when
\begin{equation} \label{eq:droplet}
\Gamma = \left\{ \sin(\pi s)\left(\cos\left(\frac{2\pi}{7}(s-1/2)\right), \sin\left(\frac{2\pi}{7}(s-1/2)\right)\right) \, : \, s \in [0,1]\right\},
\end{equation}
a droplet-shape with a corner of angle $\alpha = 2\pi/7$. Note also that in this case $\omega_{22}^+(z) = -\overline{\omega^+_{11}(-z)}$ and $\omega_{12}^+ = \omega_{21}^+ = 0$. Figure~\ref{fig:square} illustrates the limit polarizability for a unit square $\Gamma$, and suggests that in general, one can not expect the spectral measures of $K \colon H^{1/2}(\Gamma) \to H^{1/2}(\Gamma)$ to have smooth densities at $\lambda = 0$. The figures were produced by running the programs \texttt{demo17.m} and \texttt{demo17b.m}, available as a part of Johan Helsing's tutorial on the RCIP method \cite{HelTut}.
\begin{figure}
	\centering
	\includegraphics[width =0.80\linewidth]{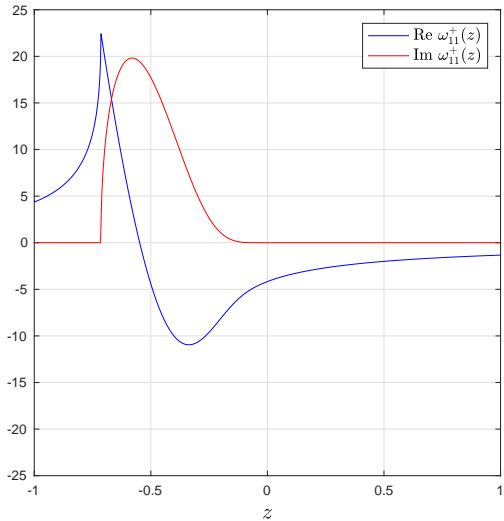}
	\caption{The limit polarizability $\omega_{11}^+(z)$, $-1 \leq z \leq 1$, when $\Gamma$ is as in \eqref{eq:droplet}, a droplet-shape with a corner of angle $\alpha = 2\pi/7$.}
	\label{fig:limpol}
\end{figure}

\begin{figure}
	\centering
	\includegraphics[width =0.80\linewidth]{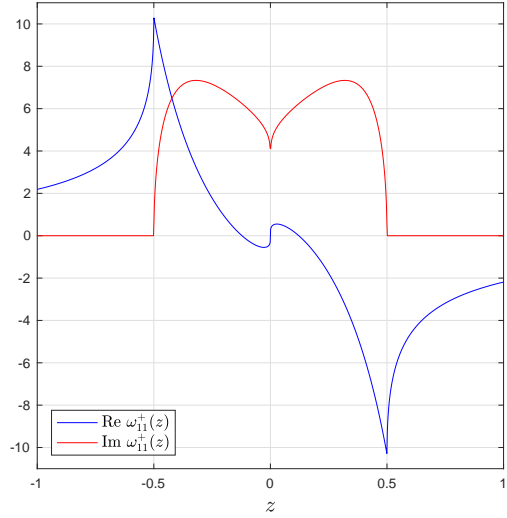}
	\caption{The limit polarizability $\omega_{11}^+(z)$, $-1 \leq z \leq 1$, when $\Gamma$ is a unit square.}
	\label{fig:square}
\end{figure}

\end{example}

\subsection*{Acknowledgements} The author is very grateful to Anne-Sophie Bonnet-Ben Dhia, Xavier Claeys, Johan Helsing, and Alexander Pushnitski for helpful discussions and suggestions. This research was supported by EPSRC Research Grant EP/S029486/1.
}
 \setcounter{theorem}{0}

\section{Layout and proof strategy}
Section~\ref{sec:prelim} contains a taxonomy of Sobolev spaces on $\mathbb{R}_+$ and their relationship with the Mellin transform.  This is needed in order to describe the Sobolev spaces $H^\eta(\Gamma)$, $0 \leq \eta \leq 1$. Particular attention is paid to the matching condition at the corner between $\Gamma_-$ and $\Gamma_+$ (Figure~\ref{fig:Gamma}), which is required for a function on $\Gamma$ to belong to $H^\eta(\Gamma)$. The condition is somewhat subtle for the critical index $\eta = 1/2$. We also a give a characterization of $H^{1/2}(\Gamma)$ in terms of single layer potentials, leading to the special $\mathcal{E}'$-norm in which $K$ is self-adjoint. Finally, we outline the basic spectral theory of $K \colon \mathcal{E}' \to \mathcal{E}'$.

In Section~\ref{sec:mellin} we first explain how to expand $K$ into a sum of Mellin operators in the vicinity of a curvilinear corner, referring to calculations made by Costabel \cite{C83} and Costabel and Stephan \cite{CS83}.  Very roughly speaking, we have the following situation: with respect to the decomposition $L^2(\Gamma) = L^2(\Gamma_-) \oplus L^2(\Gamma_+)  \oplus L^2(\Gamma_c)$, see Figure~\ref{fig:Gamma}, for suitable parametrizations of $\Gamma_{\pm}$, we that that
$$K = \begin{pmatrix}
0 & \phi K_\alpha & 0\\
\phi K_\alpha & 0 & 0 \\
0 & 0 & 0
\end{pmatrix} + \textrm{higher order terms.}$$
Here $\phi$ is a cut-off function, and $K_\alpha$ is a Mellin convolution operator whose kernel has Mellin transform
$$\widehat{K}_\alpha(z) = \frac{\sin((\pi - \alpha)z)}{\sin(\pi  z)}, \quad -1 < \mre z < 1.$$
Suppose that $(K_\alpha-\lambda) f = g$, for a function $g \in C_c^\infty([0, \infty))$ and a function $f \colon \mathbb{R}_+ \to \mathbb{C}$ which in a suitable sense has a Mellin transform on the line $\mre z = 0$. Suppose also that $\widehat{K}_\alpha(\mu) = \lambda$ for some unique $\mu$ in the strip $-1/2 < \mre \mu < \varepsilon$, $\varepsilon > 0$. From the equation $(K_\alpha-\lambda) f = g$ we deduce that the Mellin transform $\hat{f}$ of $f$ has a meromorphic extension to $-1/2 < \mre z < \varepsilon$ with poles at $0$ and $\mu$. This corresponds to the fact that $f(s) = c\phi(s)s^{-\mu} + \tilde{f}(s),$ where $c \in \mathbb{C}$ and $\tilde{f}$ is of higher regularity than $f$.

To formalize this argument we first study the transcendental equation $\widehat{K}_\alpha(z) = \lambda$. By conformal mapping theory, we show that $\widehat{K}_\alpha$ is univalent on certain strips and half-strips. In fact, $\mu_{\alpha}$, as described in \eqref{eq:intromu} and of central importance to Theorem~\ref{thm:main2}, is the inverse of a restriction of $\widehat{K}_\alpha$. Secondly, we have to perform the Mellin analysis of $K_\alpha$ very carefully, as we are interested in functions $f$ belonging precisely to the Sobolev space $H^{1/2}(\mathbb{R}_+)$ of critical index.

Assume for simplicity that $0 < \alpha < \pi$. In Section~\ref{sec:NPGamma} we apply these results to show that if $\lambda \in \widetilde{\Sigma}_{1, \alpha}$ with $\mim \lambda > 0$ and $(K - \lambda)f_\lambda = g$ for $f_\lambda \in H^{1/2}(\Gamma)$ and $g \in H^1(\Gamma)$, then $f_\lambda$ is of the form 
$$f_\lambda = c_\lambda q_{\lambda} + \tilde{f}_\lambda.$$
Here $q_{\lambda}$ is the singular function described before Theorem~\ref{thm:main2} and $\tilde{f}_\lambda \in H^1(\Gamma)$. Hence the operator $H_\lambda$ of Theorem~\ref{thm:main2} can be understood as the restriction to $H^1(\Gamma)$ of the resolvent $(K-\lambda)^{-1}$, for $\lambda \in \widetilde{\Sigma}_{1, \alpha}$, $\mim \lambda > 0$. We will then show that $\lambda \mapsto H_\lambda^{-1}$ is an analytic family of Fredholm operators for $\lambda \in \widetilde{\Sigma}_{1, \alpha} \setminus \gamma_{1, \alpha}$. Thus, by the analytic Fredholm theorem, $H_\lambda$ has a meromorphic extension to all of $\widetilde{\Sigma}_{1, \alpha} \setminus \gamma_{1, \alpha}$. Combined with the spectral theorem for $K \colon H^{1/2}(\Gamma) \to H^{1/2}(\Gamma)$, we deduce Theorems~\ref{thm:main1}, \ref{thm:main2}, and \ref{thm:main3}, except for the statement about the multiplicity of the a.c. spectrum. 

We could equally well have started by defining the operators $\widetilde{H}_\lambda$ via restriction of the resolvent to $H^1(\Gamma)$ for $\mim \lambda < 0$ (rather than $\mim \lambda > 0$). We then find that $\lambda \mapsto \widetilde{H}_\lambda$ has a meromorphic extension through the essential spectrum to $\mim \lambda \geq 0$. 

In Section~\ref{sec:spectralres} we then construct the spectral resolution of $K \colon H^{1/2}(\Gamma) \to H^{1/2}(\Gamma)$ through the formal relation
$$\frac{d}{du}E((-\infty, u)) = (2\pi i)^{-1}(H_u - \widetilde{H}_u),$$
where $E$ is the projection-valued spectral measure of $K$. This will in particular show that the absolutely continuous spectrum has multiplicity $1$.

In Section~\ref{sec:remarks} we tie up some loose ends, and briefly discuss a number of relevant open problems.
\section{Preliminaries} \label{sec:prelim}
\subsection{A description of functions on $\Gamma$} \label{subsec:desc}
Let $\Gamma_+$ and $\Gamma_-$ be two Jordan arcs with parametrizations $s \mapsto z_\pm(s)$, $s \in [0,1]$, intersecting only at $z_+(0) = z_-(0) = 0$, where the arcs make an angle $\alpha$, $0 < \alpha < 2\pi$, $\alpha \neq \pi$. We assume that $\Gamma_{\pm}$ are parametrized by arc length from the corner for, say, $s \in [0, 1/2).$  We also suppose that $\Gamma_+$ and $\Gamma_-$ are joined together by a Jordan arc $\Gamma_c$ which makes 
$$\Gamma = \Gamma_- \cup \Gamma_+ \cup \Gamma_c$$
 into a piecewise $C^3$ Jordan curve, smooth everywhere except for at the origin. See Figure~\ref{fig:Gamma}. By rotation and reflection, we may thus assume that
$$z_-(s) = e^{i\alpha}s + \frac{i}{2}\kappa_-e^{i\alpha}s^2 + O(s^3), \quad  z_+(s) = s + \frac{i}{2}\kappa_+ s^2 + O(s^3).$$
For the rest of the article, fix a positive $C^\infty(\mathbb{R}^2)$-function $\chi$ such that $\chi \equiv 1$ in a neighborhood of $0$ and such that 
$$\supp \chi \cap \Gamma \subset \bigcup_{\pm} \{z_\pm(s) \, : \, s \in [0, 1/2)\}.$$
Let $0 < s^\ast_\pm < 1/2$ be the largest numbers for which $\chi(z_{\pm}(s)) =1$ for all $s \in [0, s^\ast_\pm]$. Choose also a $C^3$ parametrization $s \mapsto z_c(s)$, $s \in [0,1]$, of the smooth part $\Gamma \setminus \bigcup_{\pm} \{z_\pm(s) \, : \, s \in [0, s^\ast_\pm/2)\}$ of $\Gamma$. 

Then any function $f$ on $\Gamma$ may be decomposed as $f = \chi f + (1-\chi) f$, where $\chi f$ is naturally associated with the pairs of functions $(f_+, f_-)$, 
$$f_\pm(s) = \chi(z_\pm(s)) f(z_\pm(s)), \quad s \in \mathbb{R}_+.$$
Similarly, we associate $(1-\chi)f$ with the function
$$f_c(s) = (1-\chi(z_c(s))) f(z_c(s)), \quad s \in \mathbb{R}_+.$$

\subsection{Mellin transforms} \label{subsec:mellin}
For appropriate functions $f \colon \mathbb{R}_+ \to \mathbb{C}$ and $z \in \mathbb{C}$, we denote the Mellin transform of $f$ by $\hat{f}(z)$,
$$\hat{f}(z) = \mathcal{M}f(z) = \int_0^\infty s^z f(s) \, \frac{ds}{s}.$$
We now list some standard lemmas that are easy to prove.
\begin{lemma} \label{lem:melliniso}
	Suppose that $F$ is analytic and rapidly decreasing in the strip $\beta < \mre z < \gamma$, that is, for every $\varepsilon > 0$ and $k, m \in \mathbb{N}_0$ we have that
	$$\left(\frac{d}{dz}\right)^k F(z) = O((1+|\mim z|)^{-m}), \quad \beta + \varepsilon < \mre z < \gamma - \varepsilon.$$
	Then $F = \hat{f}$ is the Mellin transform of a unique function $f \in C^\infty(\mathbb{R}_+)$ such that, for every $k$ and $\varepsilon > 0$,
	$$\left(t \frac{d}{dt}\right)^k f(t) = O(t^{-\beta - \varepsilon}), \quad t \to 0^+,$$
	and
	$$\left(t \frac{d}{dt}\right)^k f(t) = O(t^{-\gamma + \varepsilon}), \quad t \to \infty.$$
	More precisely, 
	$$f(t) = (\mathcal{M}^{-1} F)(t) = \frac{1}{2\pi i}\int_{\mre z = \rho} t^{-z} F(z) \, dz,$$
	for any $\beta < \rho < \gamma$.
	Conversely, if $f \in C^\infty(0,\infty)$ satisfies these conditions, then $\hat{f}$ exists and is rapidly decreasing in the strip $\beta < \mre z < \gamma$. 
\end{lemma}
We will reserve the letter $\phi$ to always mean a decreasing cut-off function $\phi \in C^\infty([0,\infty))$ around $0$, with support in $[0,1)$. That is, $1-\phi$ is compactly supported in $\mathbb{R}_+ = (0,\infty)$. 
\begin{lemma} \label{lem:mellinsing}
 For $\mu \in \mathbb{C}$, let $f(s) = \phi(s) s^{-\mu}$. Then
	$$\hat{f}(z) = \frac{\psi(z)}{z-\mu}, \quad \mre z > \mre \mu,$$
	where $\psi$ is entire and rapidly decreasing as $|\mim z| \to \infty$, and 
	$$\psi(\mu) = -\int_0^\infty \phi'(s) \, ds > 0.$$
\end{lemma}
\begin{proof}
	For $\mre z > \mre \mu$, integration by parts yields that
	$$\hat{f}(z) = -\frac{1}{z-\mu} \int_0^1 s^{z-\mu} \phi'(s) \, ds.$$
	As a function of $z$, the integral is clearly entire and, by further integration by parts, rapidly decreasing.
\end{proof}
\begin{lemma} \label{lem:mellineval}
	Let $1/2 < \eta \leq 1$ and let $f \in C_c^\infty(\mathbb{R}_+)$. Then, for any $\mu \in \mathbb{C}$ with $1/2 - \eta <  \mre \mu \leq 0$, 
	$$|\hat{f}(\mu)|^2 \lesssim \frac{1}{\mre \mu + \eta - 1/2}\int_{\mre z = 1/2 - \eta} (1+|z|^2)^{\eta}|\hat{f}(z)|^2 \, |dz| + \int_{\mre z = 1/2} |\hat{f}(z)|^2 \, |dz|.$$
\end{lemma}
\begin{proof}
	Apply Cauchy--Schwarz to the Cauchy formula
	\begin{equation*}
	\hat{f}(\mu) = \frac{1}{2\pi i}\int_{\mre z = 1/2} \frac{\hat{f}(z)}{z -\mu} \, dz -\frac{1}{2\pi i}\int_{\mre z = 1/2 - \eta} \frac{\hat{f}(z)}{z -\mu} \, dz. \qedhere
	\end{equation*}
\end{proof}

\subsection{Sobolev spaces on $\mathbb{R}_+$}
For $0 < \eta \leq 1$, the Sobolev space $H^{\eta}(\mathbb{R}_+)$ is the completion of $C_c^\infty([0,\infty))$ in the norm
$$\|f\|_{H^{\eta}(\mathbb{R}_+)}^2 = \|f\|_{L^2(\mathbb{R}_+)}^2 + |f|_{H^{\eta}(\mathbb{R}_+)}^2,$$
where, for $\eta < 1$,
$$|f|_{H^{\eta}(\mathbb{R}_+)}^2 = \int_{0}^\infty\int_{0}^{\infty} \frac{|f(s) - f(t)|^2}{|s-t|^{1+2\eta}} \, ds \, dt$$
and for $\eta =1$,
$$|f|_{H^{1}(\mathbb{R}_+)} = \|f'\|_{L^2(\mathbb{R}_+)}.$$ 
As usual, we let $H^0(\mathbb{R}_+) = L^2(\mathbb{R}_+)$. Thus $H^{\eta}(\mathbb{R}_+)$ is the restriction of $H^{\eta}(\mathbb{R})$ to $\mathbb{R}_+$.

We also need the weighted Sobolev spaces $W_0^\eta(\mathbb{R}_+)$, $0 \leq \eta \leq 1$, also called Kondratiev spaces, which are the completion of $C_c^\infty(\mathbb{R}_+)$ in the norm
$$\|f\|_{W_0^\eta(\mathbb{R}_+)}^2 = \int_{0}^\infty |f(s)|^2 \, \frac{ds}{s^{2\eta}} + |f|_{H^{\eta}(\mathbb{R}_+)}^2 \approx \int_{\mre z = 1/2 - \eta} (1+|z|^2)^{\eta} |\hat{f}(z)|^2 \, |dz|.$$
We refer to \cite{CS85} for properties of these spaces.

For $\eta > 1/2$, by a Sobolev embedding theorem, the functions $f \in H^{\eta}(\mathbb{R}_+)$ are continuous on $[0,\infty)$. With this in mind we also define, for $1/2 \leq \eta \leq 1$, the space
$$W^\eta(\mathbb{R}_+) = \mathbb{C}\phi \oplus W_0^\eta(\mathbb{R}_+),$$
where, as always, $\phi$ is a cut-off function around $0$.

Note that the classical fractional Hardy inequality \cite{KMP00} actually shows that if $\eta \neq 1/2$, then
$$\int_{0}^\infty |f(s)|^2 \, \frac{ds}{s^{2\eta}} \lesssim |f|_{H^{\eta}(\mathbb{R}_+)}^2, \quad f \in C_c^\infty(\mathbb{R}_+).$$
It follows that for $0 < \eta < 1/2$, $H^\eta(\mathbb{R}_+) = W_0^\eta(\mathbb{R}_+) \cap L^2(\mathbb{R}_+)$. For $\eta > 1/2$, if $f \in H^\eta(\mathbb{R}_+)$, then by the Hardy inequality, $f - f(0)\phi \in W_0^\eta(\mathbb{R}_+) \cap L^2(\mathbb{R}_+)$. Conversely, if $f \in W^\eta(\mathbb{R}_+) \cap L^2(\mathbb{R}_+)$, then $f \in H^\eta(\mathbb{R}_+)$.

To summarize, for $\eta \neq 1/2$ we have the following.
\begin{lemma} \label{lem:sobequiv}
For $0 < \eta < 1/2$, $f \in W_0^\eta(\mathbb{R}_+) \cap L^2(\mathbb{R}_+)$ if and only if $f \in H^\eta(\mathbb{R}_+)$ and
$$\|f\|_{H^\eta(\mathbb{R}_+)}^2 \approx \|f\|_{W_0^\eta(\mathbb{R}_+)}^2 + \|f\|_{L^2(\mathbb{R}_+)}^2.$$
 For $1/2 < \eta \leq 1$, $f \in W^\eta(\mathbb{R}_+) \cap L^2(\mathbb{R}_+)$ if and only if $f \in H^\eta(\mathbb{R}_+)$, and
 $$ \|f\|_{H^\eta(\mathbb{R}_+)}^2 \approx \|f\|_{W^\eta(\mathbb{R}_+)}^2 + \|f\|_{L^2(\mathbb{R}_+)}^2 \approx \|f - f(0)\phi\|_{W^\eta_0(\mathbb{R}_+)}^2 + |f(0)|^2 + \|f\|_{L^2(\mathbb{R}_+)}^2.$$
\end{lemma}
The case of chief concern, $\eta = 1/2$, is much more delicate. When restricted to $(0,1)$, the difference between $H^{1/2}(\mathbb{R}_+)$ and $W_0^\eta(\mathbb{R}_+)$ is no longer finite-dimensional in nature, as we just saw was the case for $\eta\neq 1/2$. In terms of the Mellin transform, we have for $f \in C_c^\infty([0, \infty))$, that \cite[Lemma~2.6]{CS85}:
\begin{equation} \label{eq:12norm}
\|f\|_{H^{1/2}(\mathbb{R}_+)}^2 \approx \int_{\mre = 1/2} |\hat{f}(z)|^2 \, |dz| + \int_{\mre z = 0} \frac{|z|^2}{1+|z|} |\hat{f}(z)|^2 \, |dz|.
\end{equation}
Here the Mellin transform on the line $\mre z = 0$ is understood as a limit from $\mre z > 0$. This interpretation is unproblematic, since the singularity at $z = 0$ is removed by the factor $|z|^2$, see Lemmas~\ref{lem:melliniso} and \ref{lem:mellinsing}.

Finally, the following result follows by interpolation.
\begin{lemma} \label{lem:sobcontain}
Let $1/2 < \eta \leq 1$. If $f \in H^{\eta}(\mathbb{R}_+)$, then $f \in W^{1/2}(\mathbb{R}_+)$ and
$$\|f\|_{W^{1/2}(\mathbb{R}_+)} \lesssim \|f\|_{H^\eta(\mathbb{R}_+)}.$$
\end{lemma}
\begin{proof}
	Suppose that $f \in H^{\eta}(\mathbb{R}_+)$. Then, by Lemma~\ref{lem:sobequiv}, $f-f(0)\phi \in W_0^\eta(\mathbb{R}_+)$ and $f-f(0)\phi \in L^2(\mathbb{R}_+) = W_0^0(\mathbb{R}_+)$. But $(W_0^\eta(\mathbb{R}_+))_{0 \leq \eta \leq 1}$ is an interpolation scale, and thus $f-f(0)\phi \in W^{1/2}_0(\mathbb{R}_+)$ with 
	$$\|f-f(0)\phi\|_{W^{1/2}_0(\mathbb{R}_+)} \lesssim \|f - f(0)\phi\|_{H^{\eta}(\mathbb{R}_+)} \lesssim \|f\|_{H^{\eta}(\mathbb{R}_+)} + |f(0)|.$$
	This yields the conclusion.
\end{proof}
\subsection{Sobolev spaces on $\Gamma$ and single layer potentials} \label{subsec:singlelayer}
We introduced the different Sobolev spaces on $\mathbb{R}_+$ in order to give the following description of $H^{\eta}(\Gamma)$, see \cite[Lemma~2.7]{CS85}. It is stated in terms of the decomposition of $f \colon \Gamma \to \mathbb{C}$ into the triple $(f_-, f_+, f_c)$ of functions supported in $[0,1)$, described in Subsection~\ref{subsec:desc}.
\begin{lemma} \label{lem:sobgamma}
$f \in H^{\eta}(\Gamma)$ if and only if $f_+ + f_- \in H^{\eta}(\mathbb{R}_+)$, $f_+ - f_- \in W_0^{\eta}(\mathbb{R}_+)$ and $f_c \in W_0^{\eta}(\mathbb{R}_+)$. Furthermore,
$$\|f\|_{H^{\eta}(\Gamma)}^2 \approx \|f_+ + f_-\|_{H^{\eta}(\mathbb{R}_+)}^2 + \|f_+ - f_-\|_{W_0^{\eta}(\mathbb{R}_+)}^2 + \|f_c\|_{W_0^{\eta}(\mathbb{R}_+)}^2.$$
\end{lemma}

We will also make use of an alternative characterization of $H^{1/2}(\Gamma)$. For a charge $f$ on $\Gamma$, consider the single layer potential
$$Sf(x) = \frac{1}{2\pi} \int_{\Gamma} \log \frac{1}{|x-y|} f(y) \, d\sigma(y), \quad x \in \mathbb{R}^2.$$
Then, by Green's formula and the jump formulas for $\partial_n^{\inte} Sf$ and $\partial_n^{\exte} Sf$,
$$\partial_n^{\inte} S f = \frac{1}{2}(I - K^\ast)f, \quad \partial_n^{\exte} S f = \frac{1}{2}(-I - K^\ast)f$$
one sees that
\begin{equation} \label{eq:Sjumpform}
\langle Sf, f \rangle_{L^2(\Gamma)} = \int_{\mathbb{R}^2} |\nabla Sf|^2 \, dx > 0, \quad 0 \neq f \in L^2_0(\Gamma),
\end{equation}
cf. \cite[Theorem~4.9]{Verchota84} and \cite[Lemma~1]{KPS07}.
Here the requirement that $f \in L^2_0(\Gamma)$, that is, $\int_{\Gamma} f \, d\sigma = 0$, arises from the need for sufficient decay of $Sf(x)$ as $x \to \infty$ in order to apply Green's formula in the exterior. 

There is a unique $\rho_0 \in L^2(\Gamma)$ such that $\int_{\Gamma} \rho_0 \, d\sigma = 1$ and $S\rho_0 \equiv D$ is constant on $\Gamma$ \cite[Section~4]{Verchota84}. This equilibrium distribution $\rho_0$ of $\Gamma$ is an element of the one-dimensional kernel of $K^\ast - I$. It may happen that $D \leq 0$, which corresponds to the fact the logarithmic capacity $\capp(\Gamma) \geq 1$. However, in this case we may simply replace $\Gamma$ by $\varepsilon \Gamma$ for a sufficiently small $\varepsilon > 0$, so that $D > 0$. By the homogeneity of the kernel of $K$, there is no loss of generality in assuming that $D > 0$. Thus we assume that $\capp(\Gamma) < 1$ throughout.

This ensures that $S \colon L^2(\Gamma) \to L^2(\Gamma)$ is a positive definite operator without kernel. Indeed, if $L^2(\Gamma) \ni g = f + c \rho_0$ for some $f \in L^2_0$ and $c \in \mathbb{C}$, then
$$\langle Sg, g \rangle_{L^2(\Gamma)} = \langle Sf, f \rangle_{L^2(\Gamma)} + D|c|^2.$$
By Verchota's results, \cite[Theorem~4.11]{Verchota84}, we also know that $S \colon L^2(\Gamma) \to H^1(\Gamma)$ is an isomorphism. By duality, so is $S \colon H^{-1}(\Gamma) \to L^2(\Gamma)$. By interpolation \cite[Theorem~15.1]{LM72}, we also find that $S \colon H^{-1/2}(\Gamma) \to H^{1/2}(\Gamma)$ is an isomorphism, and furthermore that
$$\|f\|_{H^{-1/2}(\Gamma)}^2 \approx \langle Sf, f \rangle_{L^2(\Gamma)} =: \|f\|^2_{\mathcal{E}}, \quad f \in H^{-1/2}(\Gamma).$$
Thus the completion $\mathcal{E}$ of $L^2(\Gamma)$ in the $\| \cdot \|_{\mathcal{E}}$-norm is naturally isomorphic to $H^{-1/2}(\Gamma)$.
See \cite[Theorem~14 and Corollary~15]{Per19} for more details. 

The discussion also yields that $\mathcal{E}' = S\mathcal{E}$, equipped with the norm 
$$\|Sf\|_{\mathcal{E}'} = \|f\|_{\mathcal{E}}, \quad f \in H^{-1/2}(\Gamma),$$
is naturally isomorphic to $H^{1/2}(\Gamma)$. 
\begin{lemma} \label{lem:Eprime}
	$$\|f\|_{\mathcal{E}'}^2 = \langle S^{-1}f, f \rangle_{L^2(\Gamma)}, \quad f \in H^{1/2}(\Gamma),$$
yields an equivalent norm on $H^{1/2}(\Gamma)$, $\|f\|_{H^{1/2}(\Gamma)}^2 \approx \|f\|_{\mathcal{E}'}^2$.
\end{lemma}
We shall always equip $H^{1/2}(\Gamma)$ either with the norm given in Lemma~\ref{lem:sobgamma} or the $\|\cdot\|_{\mathcal{E}'}$-norm, depending on the context. A different approach to Lemma~\ref{lem:Eprime} can be found in \cite{AMRZ17, AK16}, in which $S$ is gently but suitably modified to avoid the assumption that $\capp(\Gamma) < 1$.
\subsection{The Neumann--Poincar\'e operator} \label{subsec:spec}
The double layer potential of a charge $f$ on $\Gamma$ is given by
$$Df(x) = \frac{1}{\pi} \int_{\Gamma} \frac{\langle y - x, n_y \rangle}{|x-y|^2} \, f(y) \, d\sigma(y), \quad x \in \Gamma_{\inte} \cup \Gamma_{\exte}.$$
The NP operator is its direct value
$$Kf(x) = \frac{1}{\pi} \int_{\Gamma} \frac{\langle y - x, n_y \rangle}{|x-y|^2} \, f(y) \, d\sigma(y), \quad x \in \Gamma.$$
As $\Gamma$ is piecewise $C^3$, this latter integral exists in the ordinary sense for almost every $x \in \Gamma$, for sufficiently nice $f$. We will consider $K$ as a bounded operator $K \colon H^{s}(\Gamma) \to H^s(\Gamma)$, $0 \leq s \leq 1$. We denote by $K^\ast$ the $L^2(\Gamma)$-adjoint of $K$.

In the $\mathcal{E}$-norm, $K^\ast \colon \mathcal{E} \to \mathcal{E}$ is actually self-adjoint,
$$\langle K^\ast f, g \rangle_{\mathcal{E}} = \langle S K^\ast f, g \rangle_{L^2(\Gamma)} = \langle S f, K^\ast g \rangle_{L^2(\Gamma)} = \langle  f, K^\ast g \rangle_{\mathcal{E}}, \quad f,g \in H^{-1/2}(\Gamma).$$
This holds because of the Plemelj formula 
$$KS = SK^\ast \colon H^{-1/2}(\Gamma) \to H^{1/2}(\Gamma),$$
which also shows that  
$K^\ast \colon \mathcal{E} \to \mathcal{E}$ and $K \colon \mathcal{E}' \to \mathcal{E}'$ are unitarily equivalent.
\begin{lemma}
 $K \colon \mathcal{E}' \to \mathcal{E}'$ is self-adjoint.
\end{lemma}
We now discuss the basic spectral theory of $K \colon H^{1/2}(\Gamma) \to H^{1/2}(\Gamma)$. First of all, 
$$K+I \colon H^{1/2}(\Gamma) \to H^{1/2}(\Gamma) \textrm{ and } K-I \colon H^{1/2}(\Gamma)/\mathbb{C} \to H^{1/2}(\Gamma)/\mathbb{C}$$ 
are invertible for any Lipschitz domain; this follows by interpolation between the corresponding results on $L^2(\Gamma)$ and $H^1(\Gamma)$, due to Verchota \cite{Verchota84}. That 
$$\sigma(K, H^{1/2}(\Gamma)) \subset (-1,1]$$ 
is also well known, and follows from the invertibility of $K+I$ and the variational principle,
$$\frac{\langle K^\ast f, f \rangle_{\mathcal{E}}}{\langle f, f \rangle_{\mathcal{E}}} = \frac{\int_{\Gamma_{\exte}} |\nabla Sf |^2 \, dx - \int_{\Gamma_{\inte}} |\nabla Sf |^2 \, dx }{\int_{\Gamma_{\exte}} |\nabla Sf |^2 \, dx + \int_{\Gamma_{\inte}} |\nabla Sf |^2 \, dx},$$
valid for $f \in \mathcal{E} \simeq H^{-1/2}(\Gamma)$ such that $\int_\Gamma f \, d\sigma = \frac{1}{D}\langle f, \rho_0 \rangle_{\mathcal{E}}= 0$. This formula is established by combining Green's formula with the jump formulas for $\partial_n^{\inte} Sf$ and $\partial_n^{\exte} Sf$, cf. \eqref{eq:Sjumpform}. $\lambda = 1$ is a simple eigenvalue of $K \colon H^{1/2}(\Gamma) \to H^{1/2}(\Gamma)$, with the constants functions as eigenfunctions. 

Except for the simple eigenvalue $\lambda = 1$, the spectrum $\sigma(K, H^{1/2}(\Gamma))$ is symmetric with respect to $\lambda = 0$. In fact, there is a unitary map $V \colon \mathcal{E}'/\mathbb{C} \to \mathcal{E}'/\mathbb{C}$ such that 
$$V^{-1} K V = -K \colon \mathcal{E}'/\mathbb{C} \to \mathcal{E}'/\mathbb{C}.$$
More precisely, as elucidated in \cite[Section 5]{KPS07}, $K \colon \mathcal{E}'/\mathbb{C} \to \mathcal{E}'/\mathbb{C}$ is unitarily equivalent to the operator of the angle between the spaces of single layer potential fields in $\mathbb{R}^2$ and harmonic fields supported only in $\Gamma_{\inte}$, both equipped with the energy norm. This principle remains valid for Lipschitz domains \cite{PP14}. However, the act of harmonic conjugation (i.e. the interior Hilbert transform) yields a unitary equivalence between this operator and its additive inverse. 
We refer to \cite[Proposition~6]{KPS07} for further details in the case of a smooth boundary $\Gamma$; the considerations for Lipschitz domains are very similar, and it would lead us too far away to present them here. In any case, we will only rely on this symmetry for the eigenvalues, which have been treated in \cite[Theorem~2.1]{HKL17}.

As mentioned in the introduction, for our particular curve $\Gamma$ with a single corner of angle $\alpha$, it has been shown in \cite{BZ19, PP17} that
$$\sigma_{\ess}(K, H^{1/2}(\Gamma)) = [-|1-\alpha/\pi|, |1-\alpha/\pi|].$$


\section{Mellin analysis on $\mathbb{R}_+$} \label{sec:mellin}
For a function $f \colon \Gamma \to \mathbb{C}$ and $s \in \mathbb{R}_+$, we let
$$K_\pm f(s) = (Kf)_\pm(s), \quad K_c f(s) = (Kf)_c(s).$$
 Since $\Gamma$ is $C^3$ outside the corner, the integral kernel of $K_c$, defined on $\mathbb{R}_+ \times \Gamma$, is $C^1$ and compactly supported. Therefore
$$K_c \colon L^2(\Gamma) \to C_c^1((0,1])$$
is continuous.

To analyze $K_\pm$ we will now recall a number of calculations from \cite{C83, CS83}.
For the cut-off function $\phi$ around $0$, described in Subsection~\ref{subsec:mellin}, we have that
\begin{equation} \label{eq:Kexpansion}
K_\pm f(s) = \phi(s)K_\alpha f_\mp(s) + \phi(s)V_\pm f_\mp(s) + R_\pm f(s).
\end{equation}
Here, for a function $f$ on $\mathbb{R}_+$,
$$K_\alpha f(s) = \int_0^\infty K_\alpha(s, t) f(t) \, dt, \quad V_\pm f(s) = \int_0^\infty V_\pm(s, t) f(t)\, dt,$$
where 
$$K_\alpha(s,t) = -\frac{1}{\pi} \mim \left( \frac{e^{i\alpha}}{te^{i\alpha} - s}\right),$$
and 
$$V_\pm(s,t) = \mp \frac{1}{2\pi} \mre \left( \frac{e^{i\alpha}(\kappa_{\mp}t^2 e^{i\alpha} - 2\kappa_{\mp} st + \kappa_{\pm} s^2)}{(te^{i\alpha}-s)^2}\right).$$
$R_\pm$ are smoothing operators whose kernels have bounded derivatives in $s$,
$$R_\pm \colon L^2(\Gamma) \to C^1((0,1]).$$

Since its integral kernel is homogeneous of degree $-1$, $K_\alpha$ is a Mellin convolution operator,
$$K_\alpha f(s) = \int_0^\infty K_\alpha(s/t, 1) f(t) \, \frac{dt}{t},$$
and
$$\widehat{K}_\alpha(z) := \mathcal{M}(K_\alpha(\cdot, 1))(z) = \frac{\sin((\pi - \alpha)z)}{\sin(\pi  z)}.$$
Note that $\widehat{K}_\alpha$ is analytic and rapidly decreasing in $-1 < \mre z < 1$. 
 Let $f \in C^\infty_c([0,\infty))$. Then $\hat{f}(z)$ is defined for $\mre z > 0$, but has a meromorphic extension to $\mre z > -1$ with a simple pole at $0$, see Lemma~\ref{lem:mellinsing}. Taking this pole into account yields that $\mathcal{M}(K_\alpha f)$ also has a meromorphic extension from $0 < \mre z < 1$ to $-1 < \mre z < 1$ with 
$$\mathcal{M}(K_\alpha f)(z) = \widehat{K}_\alpha(z) \hat{f}(z) = \frac{\sin((\pi - \alpha)z)}{\sin(\pi  z)} \hat{f}(z).$$
The kernels of $V_\pm$ are homogeneous of degree $0$ and formally one has
\begin{align*}
\mathcal{M}(V_\pm f)(z) &= \widehat{V}_\pm(z) \hat{f}(z+1) \\ &:= \mp \frac{(z+i)(\kappa_{\mp} \cos((\pi-\alpha)z) + \kappa_{\pm} \cos((\pi-\alpha)(z+1))}{2\sin(\pi z)}\hat{f}(z+1).
\end{align*}
However, $\mathcal{M}(V_\pm(\cdot, 1))(z)$ does not actually exist for any $z$. Instead, a more sophisticated argument, which looks at the meromorphic continuations of the individual terms of $V_\pm$, is needed. This is carried out in \cite[Theorem~2]{C83}, from which we extract a special case.
\begin{lemma} \label{lem:Vcpct}
The operators $Z_\pm \colon L^2(\Gamma) \to H^1(\mathbb{R}_+)$ are bounded, where
$$Z_\pm f = \phi V_\pm f_{\mp}, \quad f \in L^2(\Gamma).$$ 
\end{lemma}
We also need to know that the remainder terms of the expansion \eqref{eq:Kexpansion} yield continuous functions at the corner.
\begin{lemma} \label{lem:remaindercont}
The operator $C \colon L^2(\Gamma) \to W^1_0(\mathbb{R}_+)$ is bounded, where
$$Cf = \phi(V_+f_- - V_-f_+) + (R_+ - R_-)f.$$
\end{lemma}
\begin{proof}
By Lemma~\ref{lem:Vcpct}, we already know that $C  \colon L^2(\Gamma) \to H^1(\mathbb{R}_+)$ is bounded, and actually $\supp Cf \subset [0,1)$ for every $f \in L^2(\Gamma)$. Thus, by Lemma~\ref{lem:sobequiv}, we only need to check that $Cf(0) = 0$, for arbitrary $f \in L^2(\Gamma)$. 

First of all, it is very well known that $K \colon C(\Gamma) \to C(\Gamma)$ is bounded, see for example \cite{Wendland09}. In particular, $K_+ f(0) = K_- f(0)$ for continuous $f$. Secondly, it is clear that $K_\alpha \colon W^1_0(\mathbb{R}_+) \to W^1_0(\mathbb{R}_+)$ is bounded, since $\widehat{K}_\alpha(z)$ is bounded on the line $\mre z = -1/2$. Therefore $\phi K_\alpha \colon W^1_0(\mathbb{R}_+) \to W^1_0(\mathbb{R}_+)$ is bounded as well. Thus $Cf(0) = 0$ for $f \in H^1(\Gamma)$, since in this case $f_- - f_+ \in W^1_0(\mathbb{R}_+)$, and
$$Cf = (K_+ - K_-)f - \phi K_\alpha (f_- - f_+).$$
However, since $H^1(\Gamma)$ is dense in $L^2(\Gamma)$, it follows that $Cf(0) = 0$ for arbitrary $f \in L^2(\Gamma)$.
\end{proof}

To perform the Mellin analysis of $K_\alpha$, we first have to investigate the function $\widehat{K}_\alpha$. Since $\widehat{K}_{2\pi - \alpha}(z) = -\widehat{K}_{\alpha}(z)$ it is sufficient to consider the case when $0 < \alpha < \pi$ -- unless stated otherwise, we now assume that $0 < \alpha < \pi$ in the rest of this section. Note also that $\widehat{K}_\alpha(-z) = \widehat{K}_\alpha(z)$, that $\widehat{K}_\alpha(\bar{z}) = \overline{\widehat{K}_\alpha(z)}$, and that at $z = 0$,
\begin{equation} \label{eq:Katzero}
\widehat{K}_{\alpha}(z) = 1 - \frac{\alpha}{\pi} + \frac{\alpha(\pi -\alpha)(2\pi -\alpha)}{6\pi}z^2 + O(z^3).
\end{equation}

\begin{definition}
	For $0 \leq \eta \leq 1$, we let
$$\Sigma_{\eta, \alpha} = \clos {\widehat{K}_\alpha((1/2-\eta) + i\mathbb{R})}.$$
For $\eta \neq 1/2$, we denote the interior of $\Sigma_{\eta, \alpha}$ by $\widetilde{\Sigma}_{\eta, \alpha}$.
\end{definition}
Clearly, $\Sigma_{1-\eta, \alpha} = \Sigma_{\eta, \alpha}$.
For $\eta = 1/2$ we obtain a closed interval, traversed twice,
$$\Sigma_{1/2, \alpha} = [0, 1-\alpha/\pi].$$
For $\eta \neq 1/2$, $\Sigma_{\eta, \alpha}$ is a simple closed curve, whose interior $\widetilde{\Sigma}_{\eta, \alpha} \subset \{\mre \lambda > 0\}$ contains $\Sigma_{1/2, \alpha}$, except for the point $0$,
$$(0, 1-\alpha/\pi] \subset \widetilde{\Sigma}_{\eta, \alpha} \subset \{\mre \lambda > 0\}.$$
In fact, the regions $\widetilde{\Sigma}_{\eta, \alpha}$ are increasing in $1/2 < \eta \leq 1$.
Of course, $0 \in \Sigma_{\eta, \alpha}$ for every $0 \leq \eta \leq 1$.  All of these statements, and more, are included in the (proof of the) following lemma.
\begin{lemma} \label{lem:univalent}
	$\widehat{K}_\alpha$ is a biholomorphism considered as a map
	$$\widehat{K}_\alpha \colon \{z \, : \, |\mre z| < 1/2, \; \mim z < 0\} \to \widetilde{\Sigma}_{1, \alpha} \setminus \gamma_{1,\alpha},$$
	where $\gamma_{1,\alpha} = [1-\alpha/\pi, \, \sin((\pi-\alpha)/2)]$,
	as well as a map
	$$\widehat{K}_\alpha \colon \{z \, : \, -1/2 < \mre z < 0\} \to \widetilde{\Sigma}_{1, \alpha} \setminus \gamma_{2,\alpha},$$
	where $\gamma_{2,\alpha} = [0, 1-\alpha/\pi].$
\end{lemma}
\begin{remark}
	Figure~\ref{fig:Sigma} illustrates several aspects of the following proof.
\end{remark}
\begin{proof}
First we show that $\widehat{K}_\alpha$ is a homeomorphism as a map
$$\widehat{K}_\alpha \colon \{iy \, : \, -\infty \leq y \leq 0\} \to [0, 1 - \alpha/\pi].$$
Indeed, 
$$\widehat{K}_\alpha (iy) = \frac{\sinh((\pi-\alpha)y)}{\sinh(\pi y)}$$
is a strictly increasing function for $-\infty \leq y \leq 0$, sending $-\infty$ to $0$ and $0$ to $1-\pi/\alpha$. To verify this, note that $\partial_y \widehat{K}_\alpha (iy) > 0$ is equivalent to 
$$\sinh(ut) \cosh(t) < u \sinh(t) \cosh(ut),$$
where $t = \pi y < 0$ and $0 < u = 1-\alpha/\pi < 1$. But this inequality follows easily from the fact that
$$\frac{d}{dt}\left[u \sinh(t) \cosh(ut) - \sinh(ut) \cosh(t)\right] = (u^2 -1)\sinh(ut)\sinh(t).$$

Secondly, a similar calculus argument shows that $\widehat{K}_\alpha$ is a homeomorphism as a map
$$\widehat{K}_\alpha \colon \{0 \leq x \leq 1/2\} \to [1 - \alpha/\pi, \, \sin((\pi-\alpha)/2)].$$
More precisely, $\partial_x \widehat{K}_\alpha(x) > 0$ is equivalent to the inequality
$$\sin(ut) \cos t < u\cos(ut)\sin t,$$
where $u = 1 - \alpha/\pi$ and $t = \pi x > 0$. For $0 < x \leq 1/2$, this follows easily by the fact that
$$\frac{d}{dt} \left[ u\cos(ut)\sin t  - \sin(ut)\cos(t)\right] = (1-u^2) \sin t \sin(ut).$$

Thirdly, we consider the curve $\upsilon$, where $\upsilon(t) = \widehat{K}_\alpha(1/2 - it)$, $0 \leq t \leq \infty$, noting that
$$\upsilon(t) = \sin \left(\frac{\pi - \alpha}{2}\right) \frac{\cosh((\pi - \alpha)t)}{\cosh(\pi t)} - i\cos \left(\frac{\pi - \alpha}{2}\right)\frac{\sinh((\pi - \alpha)t)}{\cosh(\pi t)}.$$
With yet another almost identical calculus argument, we see that $\mre \upsilon(t)$ is strictly decreasing in $t$. Thus $\upsilon$ is a simple curve which lies, except for the endpoints $\sin((\pi-\alpha)/2)$ and $0$, in the quadrant $\{z \, : \, \mre z > 0, \; \mim z < 0\}$.

Combining the three demonstrated facts, we see that $\widehat{K}_\alpha$ is a homeomorphism of Jordan curves $$\widehat{K}_\alpha \colon \partial \{z \, : \, 0 < \mre z < 1/2, \; \mim z < 0\} \to [0, \sin((\pi-\alpha)/2)] \cup \upsilon,$$
where the boundary on the left-hand side is understood on the Riemann sphere.
Note that $\upsilon = \Sigma_{1, \alpha} \cap \{z \, : \, \mim z \leq 0\}$ and thus that the interior of the curve $[0, \sin((\pi-\alpha)/2)] \cup \upsilon$ coincides with $\widetilde{\Sigma}_{1,\alpha} \cap \{z \, : \, \mim z < 0\}$, by the symmetries $\widehat{K}_\alpha(-z) = \widehat{K}_\alpha(z)$ and $\widehat{K}_\alpha(\bar{z}) = \overline{\widehat{K}_\alpha(z)}$. Since points $z = \varepsilon^2 - i\varepsilon$, $\varepsilon > 0$ small, are certainly mapped to the interior, cf. \eqref{eq:Katzero}, the maximum principle shows that 
$$\widehat{K}_\alpha \colon \{z \, : \, 0 < \mre z < 1/2, \; \mim z < 0\} \to \widetilde{\Sigma}_{1,\alpha} \cap \{z \, : \, \mim z < 0\}.$$
Since $\widehat{K}_\alpha$ is a homeomorphism of the Jordan curve boundaries (on the Riemann sphere) of these two domains, it follows by the converse to Carath\'eodory's theorem, cf. \cite[14.5, Exercise 10]{Conway}, that 
$$\widehat{K}_\alpha \colon \{z \, : \, 0 \leq \mre z \leq 1/2, \; \mim z \leq 0\} \to (\clos \widetilde{\Sigma_{1,\alpha}} \cap \{z \, : \mim z \leq 0\}) \setminus \{0\}$$
is a homeomorphism, and a biholomorphism of the interiors. 

The statement of the lemma is now obtained by piecing together the four demonstrated facts, using the symmetries of $\widehat{K}_\alpha$.
\end{proof}
\begin{definition}
For $0 < \alpha < \pi$, we denote the inverse of $$\widehat{K}_\alpha \colon \{z \, : \, |\mre z| < 1/2, \; \mim z < 0\} \to \widetilde{\Sigma}_{1, \alpha} \setminus \gamma_{1,\alpha}$$ by 
$$\mu_\alpha \colon \widetilde{\Sigma}_{1, \alpha} \setminus \gamma_{1,\alpha} \to \{z \, : \, |\mre z| < 1/2, \; \mim z < 0\},$$ where $\gamma_{1,\alpha} = [1-\alpha/\pi, \, \sin((\pi-\alpha)/2)]$.
\end{definition}

Note that $\mu_\alpha$ maps $(0, 1-\alpha/\pi)$ to $i\mathbb{R}_-$ and that $\mre \mu_\alpha(\lambda) > 0$ if $\mim \lambda < 0$, while $\mre \mu_\alpha(\lambda) < 0$ if $\mim \lambda > 0$, by the proof of Lemma~\ref{lem:univalent}. 

To illustrate, when $\alpha = \pi/2$, $\mu_{\pi/2}$ is easily computed, and is given by
$$\mu_{\pi/2}(\lambda) = -\frac{2i}{\lambda} \log\left( \frac{1}{2\lambda} + \sqrt{(2\lambda)^{-2} - 1} \right),$$
and is a conformal biholomorphism 
$$\mu_{\pi/2} \colon \widetilde{\Sigma}_{1, \pi/2} \setminus \gamma_{1, \pi/2} \to \{z \, : \, |\mre z| < 1/2, \; \mim z < 0\},$$ where $\gamma_{1, \pi/2} = [1/2, 1/\sqrt{2}].$ 

We are now ready to investigate $K_\alpha$ as an operator. Again, we may limit the discussion to the case when $0 < \alpha < \pi$.
\begin{lemma} \label{lem:Rinvert}
	$K_\alpha \colon H^{\eta}(\mathbb{R}_+) \to H^{\eta}(\mathbb{R}_+)$ is bounded for $0 \leq \eta \leq 1$. If $\lambda \notin \Sigma_{\eta, \alpha} \cup \Sigma_{1, \alpha}$ and $\lambda \neq \widehat{K}_\alpha(0) = 1 - \alpha/\pi$, then $K_\alpha - \lambda \colon H^\eta(\mathbb{R}_+) \to H^\eta(\mathbb{R}_+)$ is bounded from below:
	$$\|(K_\alpha - \lambda)f\|_{H^\eta(\mathbb{R}_+)} \gtrsim \|f\|_{H^\eta(\mathbb{R}_+)}.$$
\end{lemma}
\begin{proof}
For $0 \leq \eta \leq 1/2$, by Lemma~\ref{lem:sobequiv} and \eqref{eq:12norm}, we have for $f \in C^\infty_c([0,\infty))$ that
$$\|f\|_{H^\eta}^2 \approx \int_{\mre z = 1/2} |\hat{f}(z)|^2 |dz| + \int_{\mre z = 1/2 - \eta} \frac{|z|^{4\eta}}{(1+|z|)^{2\eta}} |\hat{f}(z)|^2 |dz|.$$
As stated earlier, $\mathcal{M}((K_\alpha - \lambda)f)(z) = (\widehat{K}_\alpha(z) - \lambda) \hat{f}(z)$, initially for $0 < \mre z < 1$, and for $-1 < \mre z < 1$, $z \neq 0$, by analytic continuation.  Since $\widehat{K}_\alpha(z)$ is bounded on the lines $\mre z = 1/2$ and $\mre z = 1/2 - \eta$, we see that $K_\alpha$ is bounded on $H^\eta$. Similarly, if $\lambda \notin \Sigma_{\eta, \alpha} \cup \Sigma_{1, \alpha}$, then $K_\alpha - \lambda \colon H^\eta \to H^\eta$ is bounded from below, for $|\widehat{K}_\alpha(z) - \lambda|$ is strictly bounded from below on these lines.

For $\eta > 1/2$, by Lemma~\ref{lem:sobequiv}, any $f \in H^\eta$ may be written $f = f(0)\phi + h$, where $h \in \widetilde{H}^\eta := W_0^\eta \cap L^2$, with
$$\|f\|_{H^\eta}^2 \approx |f(0)|^2 + \|h\|_{\widetilde{H}^\eta}^2 = |f(0)|^2 + \|h\|_{W_0^\eta}^2 + \|h\|_{L^2}^2.$$
The same argument as above yields the boundedness and boundedness from below of $K_\alpha - \lambda \colon \widetilde{H}^\eta \to \widetilde{H}^\eta$, assuming that $\lambda \notin \Sigma_{\eta, \alpha} \cup \Sigma_{1, \alpha}$.
In the notation of Lemma~\ref{lem:mellinsing},
$$\mathcal{M}((K_\alpha - \lambda)\phi)(z) = \frac{\psi(z)(1-\alpha/\pi - \lambda)}{z} + \psi_1(z), \quad  0 < \mre z < 1,$$
where $\psi_1(z)$, like $\psi(z)$, is analytic and rapidly decreasing in the strip $-1 < \mre z < 1$. In particular, by Lemma~\ref{lem:melliniso}, $\psi_1(z) = \hat{d}(z)$, $-1 < \mre z < 1$, for a function $d \in \widetilde{H}^\eta$, $d(t)$ satisfying the appropriate bounds at $t = 0$ and $t = \infty$. Thus, applying Lemmas~\ref{lem:melliniso} and \ref{lem:mellinsing} again,
$$(K_\alpha - \lambda) \phi = (1 - \alpha/\pi - \lambda)\phi + d \in H^\eta.$$
This in particular shows that $K_\alpha$ is bounded on $H^\eta$.

Suppose now that $(K_\alpha - \lambda)(f(0)\phi + h) = 0$, and that $\lambda \neq 1 - \alpha/\pi$, in addition to the condition that $\lambda \notin \Sigma_{\eta, \alpha} \cup \Sigma_{1, \alpha}$. Then we just saw that $(K_\alpha - \lambda)\phi \notin \widetilde{H}^\eta$. Since $(K_\alpha - \lambda)h \in \widetilde{H}^\eta$, it must be that $f(0) = 0$. But we already know that the kernel is trivial on $\widetilde{H}^\eta$ and therefore that $h = 0$ also. Hence $K_\alpha - \lambda$ has trivial kernel on $H^\eta$. Furthermore, $(K_\alpha - \lambda) \colon \widetilde{H}^\eta \to \widetilde{H}^\eta$ being bounded below, $(K_\alpha - \lambda) \widetilde{H}^\eta$ is closed in $\widetilde{H}^\eta$, and thus in $H^\eta$. Therefore $(K_\alpha - \lambda) H^\eta$ is closed as well. Combined with the trivial kernel, this demonstrates that $K_\alpha - \lambda \colon H^\eta \to H^\eta$ is bounded from below.
\end{proof}
\begin{remark}
	 Theorem~\ref{thm:mellinanalysis} shows that $K_\alpha - \lambda \colon H^1(\mathbb{R}_+) \to H^1(\mathbb{R}_+)$ is not invertible for $\lambda \in \widetilde{\Sigma}_{1,\alpha}$ with $\mim \lambda \neq 0$.
\end{remark}
\begin{lemma} \label{lem:Rinvert2}
Let $1/2 \leq \eta \leq 1$. Then $K_\alpha \colon W^\eta(\mathbb{R}_+) \to W^\eta(\mathbb{R}_+)$ is bounded.
If $\lambda \notin \Sigma_{\eta, \alpha}$ and $\lambda \neq \widehat{K}_\alpha(0) = 1 - \alpha/\pi$, then $K_\alpha - \lambda \colon W^\eta(\mathbb{R}_+) \to W^\eta(\mathbb{R}_+)$ is invertible.
\end{lemma}
\begin{proof}
	As in the proof of Lemma~\ref{lem:Rinvert}, $K_\alpha \colon W^\eta(\mathbb{R}_+) \to W^\eta(\mathbb{R}_+)$ is bounded. When $\lambda \notin \Sigma_{\eta, \alpha}$, it is also clear that  $(K_\alpha - \lambda) \colon W^\eta_0(\mathbb{R}_+) \to W^\eta_0(\mathbb{R}_+)$ is invertible, since $(\widehat{K}_\alpha(z) - \lambda)^{-1}$ is uniformly bounded from below on the line $\mre z = 1/2 - \eta$. Note that $\phi \notin W_0^\eta(\mathbb{R}_+)$, and that, as in the proof of Lemma~\ref{lem:Rinvert},
	$$(K_\alpha - \lambda)\phi = (1-\alpha/\pi - \lambda)\phi + d,$$
	where $d \in W^\eta_0(\mathbb{R}_+)$. It easily follows that $(K_\alpha - \lambda) \colon W^\eta(\mathbb{R}_+) \to W^\eta(\mathbb{R}_+)$ is invertible, if in addition $1-\alpha/\pi - \lambda \neq 0$.
\end{proof}
We now come to the main result of this section.
\begin{theorem} \label{thm:mellinanalysis}
Suppose that $\lambda \in \widetilde{\Sigma}_{1,\alpha}$ with $\mim \lambda \neq 0$. Furthermore, assume that $g \in H^1(\mathbb{R}_+)$, and that $f_\lambda \in H^{1/2}(\mathbb{R}_+)$ solves the equation  
 $$(K_\alpha - \lambda)f_\lambda = g.$$
Then $f_\lambda$ is of the form
$$f_\lambda(s) = c_\lambda \phi(s) s^{-\sgn(\mim \lambda)\mu_\alpha(\lambda)} + \tilde{f}_\lambda,$$
where $c_\lambda \in \mathbb{C}$, and $\tilde{f}_\lambda \in W^1(\mathbb{R}_+)$. Additionally, the solution satisfies the estimate
$$|c_\lambda|^2 + \|\tilde{f}_\lambda\|_{W^1(\mathbb{R}_+)}^2 \leq C_\lambda \|g\|_{H^{1}(\mathbb{R}_+)}^2.$$

For $\lambda \notin \clos \widetilde{\Sigma}_{1, \alpha}$, if we consider a solution $j_\lambda \in H^{1/2}(\mathbb{R}_+)$ to
 $$(K_\alpha - \lambda)j_\lambda = g,$$
 then $j_\lambda \in W^1(\mathbb{R}_+)$ and $\|j_\lambda\|_{W^1(\mathbb{R}_+)} \leq C_\lambda \|g\|_{H^{1}(\mathbb{R}_+)}.$
\end{theorem}
\begin{remark}
	If $\pi < \alpha < 2\pi$, then the same statement is true upon replacing $\widetilde{\Sigma}_{1,\alpha}$ with $-\widetilde{\Sigma}_{1,2\pi-\alpha}$ and the exponent $-\sgn(\mim \lambda)\mu_\alpha(\lambda)$ with $\sgn(\mim \lambda)\mu_{2\pi - \alpha}(-\lambda)$.
\end{remark}
\begin{proof}
We first assume that $\lambda \in \widetilde{\Sigma}_{1,\alpha}$ and $\mim \lambda > 0$. By (the proof of) Lemma~\ref{lem:univalent}, $z = \mu_\alpha(\lambda)$ is the unique solution to the equation $\widehat{K}_\alpha(z) = \lambda$ satisfying $-1/2 - \varepsilon \leq \mre z < - \mre \mu_\alpha(\lambda)$, for some $\varepsilon > 0$. To obtain the decomposition of $f_\lambda$, suppose first that $g$ is of the form $g = g(0)\phi + h$, where $h \in C_c^\infty(\mathbb{R}_+)$. Let 
$$w_\lambda(s) = \frac{g(0)}{1-\frac{\alpha}{\pi} - \lambda}\phi(s),$$
and let
$$v_\lambda(s) = \underbrace{\frac{d\hat{h}(\mu_\alpha(\lambda))}{\frac{d}{dz}\widehat{K}_\alpha(\mu_\alpha(\lambda))}}_{c_\lambda}\phi(s)s^{-\mu_\alpha(\lambda)},$$
where 
$$d^{-1} = \Res_{z = \mu_{\alpha}(\lambda)} \mathcal{M}(\phi(s)s^{-\mu_{\alpha}(\lambda)})(z) = -\int_0^1 \phi'(s) \, ds > 0.$$
See Lemma~\ref{lem:mellinsing}.
Since 
\begin{equation} \label{eq:mellineq}
(\widehat{K}_\alpha(z) - \lambda) \hat{f}_\lambda(z) = g(0)\hat{\phi} + \hat{h}(z),
\end{equation}
initially for almost every $z$ on the line $\mre z = 0$, we find that $\hat{f}_\lambda$ has a meromorphic extension to the strip $-1/2-\varepsilon < \mre z < -\mre \mu_\alpha(\lambda)$. Furthermore, by Lemmas \ref{lem:melliniso} and \ref{lem:mellinsing}, we see that $\hat{f}_\lambda - \hat{v}_\lambda - \hat{w}_\lambda$ is analytic and rapidly decreasing in this strip, where the individual terms are understood by meromorphic extension. Thus, by Lemma~\ref{lem:melliniso}, there is a function $u_\lambda \in C^\infty(\mathbb{R}_+)$, satisfying the appropriate bounds at $0$ and $\infty$, such that
\begin{equation} \label{eq:mellineq2}
\hat{u}_\lambda(z) = \hat{f}_\lambda(z) - \hat{v}_\lambda(z) - \hat{w}_\lambda(z), \quad -1/2 - \varepsilon < \mre z < \mre -\mu_\alpha(\lambda),
\end{equation}
where, again, we are adding up terms individually understood by continuation.
Then $u_\lambda \in W^1_0$, and by \eqref{eq:mellineq}, Lemma~\ref{lem:mellineval}, and Lemma~\ref{lem:sobequiv}, we obtain 
\begin{equation} \label{eq:ulest}
|c_\lambda|^2 + \|u_\lambda\|_{W_0^1}^2 \lesssim \|g\|_{H^1}^2,
\end{equation}
 with a constant that may depend on $\lambda$.

From \eqref{eq:mellineq2}, we see that $f_\lambda - w_\lambda \in W_0^{1/2}$. Furthermore, $u_\lambda + v_\lambda \in W_0^{1/2}$, and since the Mellin transform is an isomorphism of $W_0^{1/2}$ onto a weighted $L^2$-space on $\mre z = 0$, we conclude that $f_\lambda = v_\lambda + w_\lambda + u_\lambda$. This demonstrates the decomposition of $f_\lambda$ for this particular type of right-hand side $g = g(0)\phi + h$, $h \in C_c^\infty(\mathbb{R}_+)$, with $\tilde{f}_\lambda = w_\lambda + u_\lambda$.

We now treat the case of a general function $g \in H^1$, still assuming that $\mim \lambda > 0$. Then, by Lemma~\ref{lem:sobcontain}, $g \in W^{1/2}$ (and $g \in H^{1/2}$). Thus, by Lemma~\ref{lem:Rinvert2}, there is $k \in W^{1/2}$ such that $(K_\alpha - \lambda)k = g$. The functions $k$, $f_\lambda$, and $g$ are all elements of the space $\mathcal{H}$ defined as the completion of $C_c^\infty([0,\infty))$ in the norm
$$\|f\|^2_{\mathcal{H}} = \int_0^\infty \int_0^\infty \frac{|f(s)-f(t)|^2}{|s-t|^2} \, ds \, dt \approx \int_{\mre z = 0} \frac{|z|^2}{1+|z|} |\hat{f}(z)|^2 \, |dz|.$$
Then $(K_\alpha - \lambda)k = (K_\alpha - \lambda)f_\lambda = g$ remain valid as equations in $\mathcal{H}$. However, $K_\alpha - \lambda \colon \mathcal{H} \to \mathcal{H}$ is bounded from below, since $|\widehat{K}_\alpha(z) - \lambda|$ is bounded from below on $\mre z = 0$. Thus $f_\lambda = k \in W^{1/2}$.

Choose now a sequence $g_n = g(0)\phi + h_n$, $h_n \in C_c^\infty(\mathbb{R}_+)$, such that $g_n \to g$ in $H^1$, and thus in $W^{1/2}$. By Lemma~\ref{lem:Rinvert2}, we may then let $$f_{\lambda, n} = (K - \lambda)^{-1}g_n \in W^{1/2},$$
and $f_{\lambda, n} \to f_\lambda$ in $W^{1/2}$ as $n \to \infty$. The argument that led to the decomposition and the estimate \eqref{eq:ulest} applies equally well if we start with the information that $f_{\lambda, n} \in W^{1/2}$ (rather than $f_{\lambda, n} \in H^{1/2}$). Therefore we may decompose 
$$f_{\lambda, n}(s) = c_{\lambda, n}\phi(s)s^{-\mu_\alpha(\lambda)} + \tilde{f}_{\lambda, n}(s), \quad \tilde{f}_{\lambda, n} \in W^1,$$
with the estimate
$$|c_{\lambda,n}|^2 + \|\tilde{f}_{\lambda, n}\|_{W^1}^2 \lesssim \|g_n\|_{H^{1}}^2.$$
 Furthermore, the decomposition is unique, since $K-\lambda \colon W^{1/2} \to W^{1/2}$ is invertible, and $\phi(s) s^{-\mu_\alpha(\lambda)} \notin W^1$. Thus by applying the analysis with right-hand side $g_n - g_m$, we conclude that
$$|c_{\lambda, n} - c_{\lambda,m}|^2 + \|\tilde{f}_{\lambda, n} - \tilde{f}_{\lambda,m}\|^2_{W^1} \lesssim \|g_n - g_m\|^2_{H^1}.$$
Thus there exist $c_\lambda$ and $\tilde{f}_{\lambda} \in W^1$ such that $c_{\lambda,n} \to c_\lambda$ and $\tilde{f}_{\lambda, n} \to \tilde{f}_\lambda$ in $W^1$. Thus
$$f_\lambda = c_\lambda\phi(s)s^{-\mu_\alpha(\lambda)} + \tilde{f}_\lambda$$
and the same estimate
$$|c_\lambda|^2 + \|\tilde{f}_\lambda\|_{W^1}^2 \lesssim \|g\|_{H^{1}}^2$$
holds. 

The case when $\lambda \in \widetilde{\Sigma}_{1,\alpha}$ and $\mim \lambda < 0$ is the same, except that the solution to $\widehat{K}_\alpha(z) = \lambda$ in the strip $-1/2 < \mre z < 0$ now occurs at $z = -\mu_{\alpha}(\lambda)$. 

The case when $\lambda \notin \clos \widetilde{\Sigma}_{1,\alpha}$ is also treated analogously, except that in this case $|\widehat{K}_\alpha(z) - \lambda|$ is uniformly bounded from below on a strip $-1/2 - \varepsilon < \mre z < 1/2 + \varepsilon$.
\end{proof}
\section{The NP operator on $\Gamma$} \label{sec:NPGamma}
We now return to the study of $K \colon H^{1/2}(\Gamma) \to H^{1/2}(\Gamma)$. For simplicity we still assume that $0 < \alpha < \pi$, but will indicate the changes that need to be made for $\pi < \alpha < 2\pi$ along the way. Recall that the $\mathcal{E}'$-norm is an equivalent Hilbert space norm on $H^{1/2}(\Gamma)$ in which $K$ is self-adjoint. Thus $K - \lambda \colon H^{1/2}(\Gamma) \to H^{1/2}(\Gamma)$ is invertible whenever $\mim \lambda \neq 0$.
\begin{theorem} \label{thm:Gammaanalysis}
Suppose that $\lambda \in \widetilde{\Sigma}_{1, \alpha}$ with $\mim \lambda \neq 0$, $g \in H^{1}(\Gamma)$ and that $f_\lambda \in H^{1/2}(\Gamma)$ is the solution to
$$(K-\lambda) f_\lambda = g.$$ 
Then there is $c_\lambda \in \mathbb{C}$ and functions $h_{\lambda, \pm} \in H^{1}(\mathbb{R}_+)$ supported in $[0,1)$ such that
$$f_{\lambda, \pm}(s) := (f_\lambda)_{\pm}(s) = c_\lambda\phi(s) s^{-\sgn(\mim \lambda)\mu_\alpha(\lambda)} + h_{\lambda, \pm}(s), \quad s > 0.$$
Furthermore, $f_{\lambda, c} := (f_\lambda)_c \in W_0^{1}(\mathbb{R}_+)$.
We have the estimate
$$
|c_\lambda|^2 + \|h_{\lambda, +} + h_{\lambda, -}\|^2_{H^{1}(\mathbb{R}_+)} + \|h_{\lambda, +} - h_{\lambda, -}\|^2_{W_0^{1}(\mathbb{R}_+)} + \|f_{\lambda,c}\|_{W_0^{1}(\mathbb{R}_+)}^2 \leq C_{\lambda}\|g\|_{H^{1}(\Gamma)}^2. $$

If instead $\lambda \in -\widetilde{\Sigma}_{1, \alpha}$ with $\mim \lambda \neq 0$, then the the same statement holds with
$$f_{\lambda, \pm} = \pm c_\lambda\phi(s) s^{\sgn(\mim \lambda)\mu_\alpha(-\lambda)} + h_{\lambda, \pm}(s).$$
\end{theorem}
\begin{remark}
For $\pi < \alpha < 2\pi$, the same statement is true with $-\mu_{2\pi - \alpha}(-\lambda)$ in place of $\mu_{\alpha}(\lambda)$.
\end{remark}
\begin{proof}
Suppose that $\lambda \in \widetilde{\Sigma}_{1, \alpha}$ with $\mim \lambda \neq 0$. By the expansions \eqref{eq:Kexpansion} of $K_\pm$, Lemma \ref{lem:Vcpct} and the equation $(K - \lambda)f_\lambda = g$, we see that $\phi K_\alpha f_{\lambda, \mp} - \lambda f_{\lambda, \pm} \in H^{1}(\mathbb{R}_+)$. Furthermore, as an operator acting on $L^2([0,1])$, $W = (1-\phi)K_\alpha$ has a kernel $W(s,t)$, $W \in C^1([0, \infty) \times [0,1])$ such that $$|W(s,t)| \lesssim 1/s, \quad |\partial_s W(s,t)| \lesssim 1/s^2, \quad s \geq 1.$$
It follows that $W \colon L^2([0,1]) \to H^1(\mathbb{R}_+)$ is bounded. Hence
$$K_\alpha f_{\mp} - \lambda f_{\lambda, \pm} \in H^{1}(\mathbb{R}_+).$$
Therefore,
$$(K_\alpha - \lambda)(f_{\lambda, -} + f_{\lambda, +}) \in H^1(\mathbb{R}_+),$$
and
$$(K_\alpha + \lambda)(f_{\lambda, -} - f_{\lambda, +}) \in H^1(\mathbb{R}_+).$$
By Theorem~\ref{thm:mellinanalysis}, there are $c_\lambda \in \mathbb{C}$ and $j_1, j_2 \in W^1(\mathbb{R}_+)$ such that 
$$(f_{\lambda, -} + f_{\lambda, +})(s) = 2c_\lambda\phi(s)s^{-\sgn(\mim \lambda)\mu_\alpha(\lambda)} + j_1(s), \quad (f_{\lambda, -} - f_{\lambda, +})(s) = j_2(s).$$
Then we obtain the desired decomposition by letting 
$$h_{\lambda, +} = (j_1 - j_2)/2, \quad h_{\lambda, -} = (j_1 + j_2)/2.$$
Clearly, $h_{\lambda, \pm} \in H^{1}(\mathbb{R}_+)$, since both $j_1$ and $j_2$ belong to $W^1(\mathbb{R}_+)$ and are supported in $[0,1)$. Furthermore, we must have $h_{\lambda, +} - h_{\lambda, -} = f_{\lambda, +} - f_{\lambda, -} \in W^1_0(\mathbb{R}_+)$, since $f_{\lambda, +} - f_{\lambda, -} \in W^{1/2}_0(\mathbb{R}_+)$. Finally, to see that $f_{\lambda,c} \in W_0^1(\mathbb{R}_+)$, we only have to observe that $f_{\lambda,c} = \lambda^{-1}(K_c f_\lambda - g_c)$, and that $K_c$ maps $H^{1/2}(\Gamma)$ into $W_0^1(\mathbb{R}_+)$. The existence of the estimate now follows from the closed graph theorem, since we know that $(K-\lambda)^{-1} \colon H^{1}(\Gamma) \to H^{1/2}(\Gamma)$ is bounded.

The case when $\lambda \in -\widetilde{\Sigma}_{1,\alpha}$ follows along the same lines.
\end{proof}

The following result on the Fredholm property of $K-\lambda \colon H^{\eta}(\Gamma) \to H^{\eta}(\Gamma)$ is standard, but a proof is included in lieu of an exact reference.
\begin{theorem} \label{thm:fredholm}
Let $1/2 < \eta \leq 1$, and suppose that $\lambda \notin -\Sigma_{\eta, \alpha} \cup \Sigma_{\eta, \alpha}$ and $\lambda \neq \pm(1 - \alpha/\pi)$. Then $K-\lambda \colon H^{\eta}(\Gamma) \to H^{\eta}(\Gamma)$ is Fredholm.
\end{theorem}
\begin{proof}
Let $\phi_1 \in C^\infty([0,\infty))$ be a cut-off function around $0$ with sufficiently small support. Define $W \colon H^{\eta}(\Gamma) \to H^{\eta}(\Gamma)$ by
$$W_\pm f = (Wf)_{\pm} = \phi_1 K_\alpha \phi_1 f_\mp, \quad W_c f = (Wf)_c \equiv 0.$$
Note that the continuity of $Wf$ at the corner follows from the boundedness of $K_\alpha \colon W^\eta_0(\mathbb{R}_+) \to W^\eta_0(\mathbb{R}_+)$.
Then, by the expansion \eqref{eq:Kexpansion}, Lemma~\ref{lem:Vcpct}, Lemma~\ref{lem:remaindercont}, and the compact embedding of $H^\eta(\Gamma)$ into $L^2(\Gamma)$, we know that 
$$K - W \colon H^{\eta}(\Gamma) \to H^{\eta}(\Gamma)$$ 
is compact. To prove the lemma, we will show that $R \colon H^{\eta}(\Gamma) \to H^{\eta}(\Gamma)$ is a parametrix for $W - \lambda$, where
$$(Rg)_{\pm} =  \frac{1}{2}\left(\phi_1(K_\alpha - \lambda)^{-1}\phi_1 (g_+ + g_-) \pm \phi_1(K_\alpha + \lambda)^{-1} \phi_1 (g_- - g_+)\right) - \frac{1}{\lambda}(1-\phi_1^2)g_\pm$$
and $(Rg)_c = -\frac{1}{\lambda}g_c$. Here we mean $(K_\alpha \pm \lambda)^{-1}$ as an inverse on $W^{\eta}(\mathbb{R}_+)$, as in Lemma~\ref{lem:Rinvert2}, noting that a compactly supported function $f$ belongs to $W^\eta(\mathbb{R}_+)$ if and only if it belongs to $H^{\eta}(\mathbb{R}_+)$. 

Obviously, $(R(W-\lambda))_c f = ((W-\lambda)R)_c f = f_c$, and a standard (somewhat tedious) calculation shows that 
$$((W-\lambda)R)_+ f \pm ((W-\lambda)R)_- f = f_+ \pm f_- + C_\pm^{r} f$$
and
$$(R(W-\lambda))_+ f \pm (R(W-\lambda))_- f = f_+ \pm f_- + C_\pm^{l} f$$
where $C_+^r, C_+^l \colon H^{\eta}(\Gamma) \to H^{\eta}(\mathbb{R}_+)$ and $C_-^r, C_-^l \colon H^{\eta}(\Gamma) \to W_0^{\eta}(\mathbb{R}_+)$ are compact operators. Thus $(W-\lambda)R-I$ and $R(W-\lambda) - I$ are compact operators. That is, $W-\lambda$, and therefore $K-\lambda$, are Fredholm.
\end{proof}
\begin{remark}
	An inspection of the proofs of Lemma~\ref{lem:Rinvert2} and Theorem~\ref{thm:fredholm} shows that $K-\lambda \colon H^\eta(\Gamma) \to H^\eta(\Gamma)$ is also Fredholm for $\lambda = \pm (1-\alpha/\pi)$, $1/2 < \eta \leq 1$. However, we shall not need this.
\end{remark}
We introduce for $\lambda \in \widetilde{\Sigma}_{1,\alpha} \setminus \gamma_{1,\alpha}$ the function $q_{\lambda, \alpha} \in L^2(\Gamma)$ such that
$$(q_{\lambda, \alpha})_{\pm}(s) = \phi(s)s^{-\mu_\alpha(\lambda)}, \quad (q_{\lambda, \alpha})_{c} \equiv 0.$$
Note that $-1/2 < \mre \mu_{\alpha}(\lambda) < 1/2$ for such $\lambda$, by design. Similarly, for $\lambda \in -(\widetilde{\Sigma}_{1,\alpha} \setminus \gamma_{1,\alpha})$
we define $q_{\lambda, \alpha}$ so that 
$$(q_{\lambda, \alpha})_{\pm}(s) = \pm \phi(s)s^{\mu_\alpha(-\lambda)}, \quad (q_{\lambda, \alpha})_{c} \equiv 0.$$

For each $\lambda \in -(\widetilde{\Sigma}_{1,\alpha} \setminus \gamma_{1,\alpha}) \cup (\widetilde{\Sigma}_{1,\alpha} \setminus \gamma_{1,\alpha})$, we then consider the operator
$$T_\lambda \colon \mathbb{C} \oplus H^1(\Gamma) \to H^1(\Gamma)$$
such that
$$T_\lambda(c, f) = (K-\lambda)(cq_{\lambda,\alpha} + f).$$
\begin{lemma} \label{lem:anal}
$T_\lambda \colon \mathbb{C} \oplus H^1(\Gamma) \to H^1(\Gamma)$ is bounded, and the operator-valued function $\lambda \mapsto T_\lambda$ is analytic in $\widetilde{\Sigma}_{1,\alpha} \setminus \gamma_{1,\alpha}$ and $-(\widetilde{\Sigma}_{1,\alpha} \setminus \gamma_{1,\alpha})$.
\end{lemma}
\begin{proof}
	We only have to check that $\lambda \mapsto (K-\lambda)q_{\lambda, \alpha} \in H^1(\Gamma)$ is well-defined and analytic. By the expansion \eqref{eq:Kexpansion},
	$$[(K-\lambda)q_{\lambda, \alpha}]_\pm = \phi K_\alpha (q_{\lambda, \alpha})_\mp - \lambda \phi (q_{\lambda, \alpha})_\pm - \lambda(1-\phi)(q_{\lambda, \alpha})_\pm + \phi V_\pm (q_{\lambda, \alpha})_\mp + R_\pm q_{\lambda, \alpha}. $$
	Since $\lambda \mapsto q_{\lambda, \alpha} \in L^2(\Gamma)$ and $\lambda \mapsto (1-\phi)(q_{\lambda, \alpha})_\pm \in W^1_0(\mathbb{R}_+)$ are analytic, it is by Lemmas~\ref{lem:Vcpct} and \ref{lem:remaindercont} sufficient to demonstrate that
			$$\lambda \mapsto \phi (K_\alpha + \lambda)( (q_{\lambda, \alpha})_- - (q_{\lambda, \alpha})_+) \in W^1_0(\mathbb{R}_+)$$
	and
		$$\lambda \mapsto \phi (K_\alpha - \lambda)( (q_{\lambda, \alpha})_- + (q_{\lambda, \alpha})_+) \in H^1(\mathbb{R}_+) $$
		are analytic. 
		
	The question therefore boils down to proving that $\lambda \mapsto (K_\alpha -\lambda) v_\lambda \in W^1_0(\mathbb{R}_+)$ is analytic for $\lambda \in \widetilde{\Sigma}_{1,\alpha} \setminus \gamma_{1,\alpha}$, where $v_\lambda(s) = \phi(s)s^{-\mu_\alpha(\lambda)}$. Integration by parts yields that
	$$\hat{v}_\lambda(z) = \frac{\psi_\lambda(z)}{z - \mu_{\alpha}(\lambda)}, \quad \mre z > \mre \mu_{\alpha}(\lambda),$$
	where $\psi_\lambda(z) = -\int_0^1 s^{z-\mu_{\alpha}(\lambda)} \phi'(s) \, ds.$ Therefore
	\begin{equation} \label{eq:anal}
	\mathcal{M}((K_\alpha - \lambda)v_\lambda)(z) = \frac{\widehat{K}_\alpha(z) - \lambda}{z - \mu_{\alpha}(\lambda)} \psi_\lambda(z)
	\end{equation}
	has an analytic extension to $-1 < \mre z < 1$ which is rapidly decreasing. By Lemma~\ref{lem:melliniso}, there is $u_\lambda \in C^\infty(0,\infty)$, satisfying the appropriate bounds at $0$ and $\infty$, such that
	$$\hat{u}_\lambda = \mathcal{M}((K_\alpha - \lambda)v_\lambda)(z), \quad -1 < \mre z < 1.$$
	In particular this holds for $\mre z = 1/2$, and thus $u_\lambda = (K_\alpha - \lambda) v_\lambda$. On the other hand, it is clear from \eqref{eq:anal} that $u_\lambda \in W^1_0(\mathbb{R}_+)$. 
	
	Furthermore, for $\lambda$ in a sufficiently small neighborhood of a fixed $\lambda_0 \in \widetilde{\Sigma}_{1,\alpha} \setminus \gamma_{1,\alpha}$, $\hat{u}_\lambda$ and $\frac{d}{d\lambda} \hat{u}_\lambda$ are rapidly decreasing in a strip $|\mre z + 1/2|< \varepsilon$, $\varepsilon > 0$ sufficiently small, with uniform estimates. Cauchy's formula then implies that
	 $$\lambda \to \hat{u}_\lambda \in L^2(\mre z = -1/2, \, (1+|z|^2) |dz|)$$ is analytic. This is of course the same as saying that $$\lambda \mapsto (K_\alpha - \lambda) v_\lambda \in W^1_0(\mathbb{R}_+)$$
	 is analytic.
\end{proof}

By slight abuse of notation, we will identify the pair $T_\lambda^{-1}g = (c_\lambda, \tilde{f}_\lambda)$ with the function $c_\lambda q_{\lambda, \alpha} + \tilde{f}_\lambda \in L^2(\Gamma)$, 
$$T_\lambda^{-1}g = c_\lambda q_{\lambda, \alpha} + \tilde{f}_\lambda.$$

\begin{theorem} \label{thm:Tinvert}
$T_\lambda \colon \mathbb{C} \oplus H^1(\Gamma) \to H^1(\Gamma)$ is invertible for every $\lambda \in \widetilde{\Sigma}_{1,\alpha} \setminus \gamma_{1,\alpha}$ with $\mim \lambda \geq 0$, except for $\lambda = \rho_n$ belonging to a (possibly finite) sequence $\{\rho_n\} \subset (0, 1-\alpha/\pi)$ that can only accumulate at $0$ or $(1-\alpha/\pi)$. $T_{\rho_n}$ is Fredholm with index $0$ for every such exceptional point. The operator-valued function $\lambda \mapsto T_\lambda^{-1}$ has a meromorphic extension to all of $\widetilde{\Sigma}_{1,\alpha} \setminus \gamma_{1,\alpha}$, having a pole at each $\rho_n$ in addition to possible poles for $\mim \lambda < 0$. At each pole, the negative order operator coefficients in the Laurent expansion are of finite rank. Outside of the poles, it holds that
\begin{equation} \label{eq:Tinverse}
(K-\lambda)T_\lambda^{-1}g = g, \quad g \in H^{1}(\Gamma).
\end{equation}

Similarly, $T_\lambda$ is invertible for every $\lambda \in -(\widetilde{\Sigma}_{1,\alpha} \setminus \gamma_{1,\alpha})$ with $\mim \lambda \geq 0$, except for a sequence $\{\rho_n'\} \subset (-(1-\alpha/\pi), 0)$ that can only accumulate at $0$ or $-(1-\alpha/\pi)$. All the conclusions of the previous case hold.
\end{theorem}
\begin{proof}
We treat only the case when $\lambda \in \widetilde{\Sigma}_{1,\alpha} \setminus \gamma_{1,\alpha}$, as the other statement is proven in exactly the same way. 

First of all, we observe that $T_\lambda$ is invertible for $\mim \lambda > 0$. For suppose that $T_\lambda (c, f) = 0$ for some $c \in \mathbb{C}$, $f \in H^1(\Gamma)$. Then $h = cq_{\lambda, \alpha} + f \in H^{1/2}(\Gamma)$ satisfies $(K-\lambda)h = 0$. However, $K-\lambda \colon H^{1/2}(\Gamma) \to H^{1/2}(\Gamma)$ is invertible for all $\mim \lambda \neq 0$, and thus $h = 0$, whence $c = f = 0$. To see that any $g \in H^1(\Gamma)$ is in the range of $T_\lambda$, note, again by the invertibility of $K-\lambda \colon H^{1/2}(\Gamma) \to H^{1/2}(\Gamma)$, that there is $f_\lambda \in H^{1/2}(\Gamma)$ satisfying $(K-\lambda)f_\lambda = g$. But by Theorem~\ref{thm:Gammaanalysis}, this solution is of the form $f_\lambda = c_\lambda q_{\lambda, \alpha} + \tilde{f}_\lambda$, with $\tilde{f}_\lambda \in H^1(\Gamma)$.

Secondly, by Theorem~\ref{thm:fredholm} and Lemma~\ref{lem:anal}, $ \lambda \mapsto T_\lambda$, $\lambda \in \widetilde{\Sigma}_{1,\alpha} \setminus \gamma_{1,\alpha}$, is an analytic family of Fredholm operators. Since $T_\lambda$ is invertible for at least one $\lambda$, we have that $\ind T_\lambda = 0$ for every $\lambda$. Furthermore, by the analytic Fredholm theorem \cite[Corollary~8.4]{GGK90}, the set $\{w_n\}$, of points where $T_{w_n}$ fails to be invertible, is discrete in $\widetilde{\Sigma}_{1,\alpha} \setminus \gamma_{1,\alpha}$, and $T^{-1}_\lambda$ is meromorphic, with finite rank negative order operator coefficients at its poles. Equation \eqref{eq:Tinverse} remains valid all $\lambda \neq w_n$, by analyticity.
\end{proof}

As we shall now speak of $K \colon H^{1/2}(\Gamma) \to H^{1/2}(\Gamma)$ as a self-adjoint operator, it is in the sequel important to think of $H^{1/2}(\Gamma)$ as being equipped with the $\mathcal{E}'$-norm. Let $K = \int_{-1}^1 \lambda \, dE(\lambda)$ be the spectral decomposition of $K$. Recall that for $f,h \in H^{1/2}(\Gamma)$, we denote by 
$$\nu_{f,h}(\Lambda) = \langle E(\Lambda) f, h \rangle_{\mathcal{E}'}, \quad \Lambda \subset [-1,1] \textrm{ Borel},$$
the associated spectral measure.

\begin{lemma} \label{lem:eigenvalues}
The sequences $\{\rho_n'\}$ and $\{\rho_n\}$ of Theorem~\ref{thm:Tinvert} are precisely those  eigenvalues of $K \colon H^{1/2}(\Gamma) \to H^{1/2}(\Gamma)$ that are embedded in $(-(1-\alpha/\pi), 0) \cup (0, 1-\alpha/\pi)$. Therefore $\rho_n' = -\rho_n$, if ordered correctly. Each $\rho_n$ is in fact a simple pole of $T_\lambda^{-1}$, and for $g, h \in H^1(\Gamma)$,
\begin{equation} \label{eq:resformula}
\nu_{g,h}(\{\rho_n\}) = -\Res_{\lambda = \rho_n} \langle T^{-1}_{\lambda}g, h \rangle_{\mathcal{E}'}.
\end{equation}
A similar statement holds for $\rho_n'$.
\end{lemma}
\begin{proof}
	Suppose that $\lambda \in (0, 1-\alpha/\pi)$ is an eigenvalue of $K$ with eigenvector $f \in H^{1/2}(\Gamma)$, $\|f\|_{\mathcal{E}'} = 1$. Then $\nu_{f,f}$ is a point mass at $\lambda$, $\nu_{f,f}(\{\lambda\})=1$. By density, we can find $g \in H^1(\Gamma)$ such that $g = f + h$, with $\|h\|_{\mathcal{E}'} < 1/4$. Then $\nu_{g,g}(\{\lambda\}) \geq 1/2$, since
	$$\nu_{g,g} = \nu_{f,f} + 2\mre \nu_{f,h} + \nu_{h,h}.$$
	Hence,
\begin{equation} \label{eq:resolvconv}
y\langle (K-(\lambda + iy))^{-1} g, g \rangle_{\mathcal{E}'} = \int_{-1}^1 \frac{y \, d\nu_{g,g}(x)}{x-(\lambda+iy)} \to i\nu_{g,g}(\{\lambda\}) \neq 0,
\end{equation}
	as $y \to 0^+$, by the dominated convergence theorem.
	On the other hand, if $\lambda \notin \{\rho_n\}$, then
	$$(K-(\lambda + iy))^{-1} g = T_{\lambda+iy}^{-1} g \to T^{-1}_\lambda g, \quad y \to 0^+,$$
	with convergence in $\mathbb{C} \oplus H^1(\Gamma)$, and therefore in $L^2(\Gamma)$. Thus, since $S^{-1}g \in L^2(\Gamma)$,
$$
	\langle (K-(\lambda + iy))^{-1} g, g \rangle_{\mathcal{E}'} \to \langle T^{-1}_\lambda g, S^{-1}g \rangle_{L^2(\Gamma)},
$$
as $y\to0^+$, a contradiction. Therefore $\lambda \in \{\rho_n\}$.

Conversely, suppose that $\lambda = \rho_n$ for some $n$. Then $T^{-1}_{\lambda'}$ has a pole at $\lambda' = \lambda$ of some order $n \geq 1$, and thus we can arrange $f, h \in H^1(\Gamma)$ such that
$$\mim \langle y^n(K-(\lambda + iy))^{-1} f, h \rangle_{\mathcal{E}'} \geq 1$$
for $y > 0$ small. Since 
$$\mre \nu_{g,h}(\{\lambda\}) = \lim_{y\to0^+} \mim \langle y(K-(\lambda + iy))^{-1} f, h \rangle_{\mathcal{E}'}$$
it must be that $n=1$ and $E(\{\lambda\}) \neq 0$, that is, $\lambda$ is an eigenvalue of $K \colon H^{1/2}(\Gamma) \to H^{1/2}(\Gamma)$. 

The formula for $\nu_{g,g}(\{\rho_n\})$ now follows by comparing \eqref{eq:resolvconv} with the fact that
$$y\langle (K-(\lambda + iy))^{-1} g, g \rangle_{\mathcal{E}'} = - i \langle \Res_{\lambda = \rho_n} T^{-1}_{\lambda}g, S^{-1}g \rangle_{L^2(\Gamma)} + O(y).$$
The more general formula \eqref{eq:resformula} follows by a polarization argument.

An identical argument works for the eigenvalues in $(-(1-\alpha/\pi),0)$. The fact that $\{\rho_n'\} = \{-\rho_n\}$ then follows by the symmetry of the spectrum described in Subsection~\ref{subsec:spec}.
\end{proof}
\begin{remark}
	Of course the choice to define $q_{\lambda, \alpha}$ to agree with the decomposition of Theorem~\ref{thm:Gammaanalysis} for $\mim \lambda > 0$, rather than $\mim \lambda < 0$, was arbitrary. We could instead have let
	$$(\tilde{q}_{\lambda, \alpha})_{\pm}(s) = \phi(s)s^{\mu_\alpha(\lambda)}, \quad (\tilde{q}_{\lambda, \alpha})_{c} \equiv 0, \quad \lambda \in \widetilde{\Sigma}_{1,\alpha} \setminus \gamma_{1,\alpha},$$
	and
	$$(\tilde{q}_{\lambda, \alpha})_{\pm}(s) = \pm \phi(s)s^{-\mu_\alpha(-\lambda)}, \quad (\tilde{q}_{\lambda, \alpha})_{c} \equiv 0, \quad \lambda \in -(\widetilde{\Sigma}_{1,\alpha} \setminus \gamma_{1,\alpha}).$$
	Letting
	$$\widetilde{T}_\lambda(c, f) = (K-\lambda)(c\tilde{q}_{\lambda, \alpha} + f),$$
	the same proof yields that inside $-(\widetilde{\Sigma}_{1,\alpha} \setminus \gamma_{1,\alpha}) \cup (\widetilde{\Sigma}_{1,\alpha} \setminus \gamma_{1,\alpha})$, $\lambda \to \widetilde{T}_\lambda^{-1}$ has a meromorphic extension from $\mim \lambda < 0$ to $\mim \lambda > 0$. The real poles coincide with the eigenvalues of $K \colon H^{1/2}(\Gamma) \to H^{1/2}(\Gamma)$.
\end{remark}
We have now essentially proven Theorem~\ref{thm:main1} and Theorem~\ref{thm:main2}. It is obvious how to modify the construction for $\pi < \alpha < 2\pi$ based on the remark after Theorem~\ref{thm:Gammaanalysis}. 
\begin{proof}[Proof of Theorems~\ref{thm:main1} and \ref{thm:main2}]
The desired conclusions are immediate from Theorem~\ref{thm:Tinvert} and Lemma~\ref{lem:eigenvalues}, with $H_\lambda = T_\lambda^{-1}$. It only remains to note that the eigenvalues of $K\colon H^{1/2}(\Gamma) \to H^{1/2}(\Gamma)$ that lie outside $[-|1-\alpha/\pi|, |1-\alpha/\pi|]$ can only accumulate at $\pm |1-\alpha/\pi|$, since we know already that $\sigma_{\ess}(K, H^{1/2}(\Gamma)) = [-|1-\alpha/\pi|, |1-\alpha/\pi|]$.
\end{proof}

We now prove Theorem~\ref{thm:main3}, except for the statement on multiplicity. Let $\{\lambda_n\}$ denote the eigenvalues of $K \colon H^{1/2}(\Gamma) \to H^{1/2}(\Gamma)$.
\begin{theorem} \label{thm:spectrum1}
The absolutely continuous part of the spectrum of $K \colon H^{1/2}(\Gamma) \to H^{1/2}(\Gamma)$ is given by
$$\sigma_{\ac}(K) = [-|1-\alpha/\pi|, |1-\alpha/\pi|].$$
There is no singular continuous spectrum. 

Furthermore, if $g,h \in H^1(\Gamma)$, then the spectral measure 
$$\nu_{g,h} = \tau \, d\lambda + \sum_{n} \tau_n \delta_{\lambda_n}$$
has a continuous density $\tau$ which is real analytic in $(-|1-\alpha/\pi|, 0)$ and $(0, |1-\alpha/\pi|)$.
\end{theorem}
\begin{proof}
As always, we may restrict ourselves to the case when $0 < \alpha < \pi$. 
Since we know that $\sigma_{\ess}(K, H^{1/2}(\Gamma)) = [-(1-\alpha/\pi), 1-\alpha/\pi]$ and that the eigenvalues $\{\lambda_n\}$ can only accumulate to $0$ and $\pm (1-\alpha/\pi)$, it is sufficient to prove that there is no singular continuous spectrum. Suppose to the contrary that $K$ had some singular continuous part in its spectrum. Then there would exist a closed Lebesgue null set $E$ with
$$E \subset (-(1-\alpha/\pi),0)\cup(0, 1-\alpha/\pi), \quad E \cap \{\lambda_n\} = \emptyset,$$ and $f \in H^{1/2}(\Gamma)$, $\|f\|_{\mathcal{E}'} = 1$, such that $\supp \nu_{f,f} \subset E$ and $\nu_{f,f}(E) = 1$. By density, there is $g \in H^1(\Gamma)$ such that $g = f + h$, where $\|h\|_{\mathcal{E}'} < 1/4$. Then
$$\nu_{g,g}(E) = (\nu_{f,f} + 2\mre \nu_{f,h} + \nu_{h,h})(E) \geq 1/2.$$
By inner regularity there is thus a closed Lebesgue null set $F \subset E$ such that $\nu_{g,g}(F) > 0$. However, by the theorem of de la Vall\'ee-Poussin, the singular part of $\nu_{g,g}$ is supported on the set of $\lambda \in [-1,1]$ such that
$$\mim \, \langle (K-(\lambda + iy))^{-1} g, g \rangle_{\mathcal{E}'} = \mim \int_{-1}^1 \frac{d\nu_{g,g}(x)}{x-(\lambda+iy)} = \int_{-1}^1 \frac{y \, d\nu_{g,g}(x)}{(x-\lambda)^2 + y^2} \to \infty,$$
as $y \to 0^+$. However, if $\lambda \in F$, then, as in the proof of Lemma~\ref{lem:eigenvalues},
$$\langle (K-(\lambda + iy))^{-1} g, g \rangle_{\mathcal{E}'} \to \langle T^{-1}_\lambda g, S^{-1}g \rangle_{L^2(\Gamma)},$$
as $y \to 0^+$. This is a contradiction, and we conclude that $K$ has absolutely continuous spectrum outside of its eigenvalues.

If $g \in H^1(\Gamma)$, we now know that the spectral measure may be decomposed as 
$$\nu_{g,g} = \tau \, d\lambda + \sum_n \tau_n \delta_{\lambda_n},$$
where $\sum_n \tau_n = \sum_n \nu_{g,g}(\{\lambda_n\}) < \infty$ and $\int_{-1}^1 \tau \, d\lambda < \infty$. Then by Fatou's theorem and  \eqref{eq:resformula}, 
\begin{align*}
\pi \tau(\lambda) &= \lim_{y\to0^+} \int_{-1}^1 \frac{y \, \tau(x) \, dx}{(x-\lambda)^2 + y^2} \\ 
&= \lim_{y\to0^+} \left(\int_{-1}^1 \frac{y \, d\nu_{g,g}(x)}{(x-\lambda)^2 + y^2} + \mim\sum_{\lambda_n \in (0, 1-\alpha/\pi)} \frac{\nu_{g,g}(\{\lambda_n\})}{(\lambda + iy) - \lambda_n}\right) 
\\&=  \lim_{y\to0^+} \mim\left( \, \langle (K-(\lambda + iy))^{-1} g, g \rangle_{\mathcal{E}'} - \sum \frac{\Res_{\lambda = \lambda_n} \langle T^{-1}_{\lambda}g, S^{-1}g \rangle_{L^2(\Gamma)}}{(\lambda+iy) - \lambda_n}\right) \\ &= \mim \left( \, \langle T^{-1}_\lambda g, S^{-1}g \rangle_{L^2(\Gamma)} - \sum_{\lambda_n \in (0, 1-\alpha/\pi)} \frac{\Res_{\lambda = \lambda_n} \langle T^{-1}_{\lambda}g, S^{-1}g \rangle_{L^2(\Gamma)}}{\lambda - \lambda_n}\right)
\end{align*}
for almost every $\lambda$. On the right-hand side we have subtracted off every pole from $T_\lambda^{-1}g$ in $(0,1-\alpha/\pi)$, showing that $\tau(\lambda)$ is indeed real analytic in this interval. A similar argument holds for $(-(1-\alpha/\pi), 0)$. For the general case of a spectral measure $\nu_{g,h}$, $g,h \in H^1(\Gamma)$, we employ the polarization identity.
\end{proof}

\section{The spectral resolution} \label{sec:spectralres}
To prove that the a.c. spectrum has multiplicity 1, and thereby finish the proof of Theorem~\ref{thm:main3}, we will in this section construct the diagonalization of $K \colon H^{1/2}(\Gamma) \to H^{1/2}(\Gamma)$ somewhat explicitly. We will assume that $0 < \alpha < \pi$, the case of $\pi < \alpha < 2\pi$ being analogous. 
\begin{lemma} \label{lem:eigspacedim}
Suppose that $\lambda \in (-(1-\alpha/\pi),0)\cup(0, 1-\alpha/\pi)$ is not an eigenvalue of $K \colon H^{1/2}(\Gamma) \to H^{1/2}(\Gamma)$. Then $K-\lambda \colon L^2(\Gamma) \to L^2(\Gamma)$ is Fredholm, has full range, and
$$\dim \ker (K-\lambda)|_{L^2(\Gamma)} = 1.$$
\end{lemma}
\begin{proof}
Since $T_\lambda \colon \mathbb{C} \oplus H^1(\Gamma) \to H^1(\Gamma)$ is invertible, we know that 
$$\dim \ker (K-\lambda)|_{H^1(\Gamma)} = 0, \quad \codim (K-\lambda)H^1(\Gamma) = 1.$$
Thus the Fredholm operator $K^\ast-\lambda \colon H^{-1}(\Gamma) \to H^{-1}(\Gamma)$ has full range and
$$\dim \ker (K^\ast-\lambda)|_{H^{-1}(\Gamma)} = 1.$$
The result follows, since $S \colon H^{-1}(\Gamma) \to L^2(\Gamma)$ is an isomorphism, and
$$(K-\lambda) =  S(K^\ast - \lambda)S^{-1} \colon L^2(\Gamma) \to L^2(\Gamma),$$
by the Plemelj formula.
\end{proof}
\begin{theorem} \label{thm:spectrum2}
	Let $I = (-(1-\alpha/\pi), 1-\alpha/\pi)$, and let $\{\lambda_n\}_{n \in N}$ denote the eigenvalues of $K \colon H^{1/2}(\Gamma) \to H^{1/2}(\Gamma)$, indexed by a countable set $N$, and let $\{e_n\}_{n \in N}$ denote a corresponding set of eigenvectors, orthonormal in the $\mathcal{E'}$-norm. There is a function $e \colon I \to L^2(\Gamma)$ satisfying the following: 
	\begin{itemize} 
		\item For every $\lambda \in I$, $\lambda \notin \{\lambda_n\} \cup \{0\}$, we have that $e(\lambda) \in \ker (K-\lambda)|_{L^2(\Gamma)}$. 
		\item If we for $f \in H^1(\Gamma)$ let
		$$Uf(\lambda) = \langle f, e(\lambda) \rangle_{\mathcal{E}'}, \quad Uf(n) = \langle f, e_n \rangle_{\mathcal{E}'}, \quad \lambda \in I, \; n\in N,$$
		then $U$ extends to a unitary map $U \colon H^{1/2}(\Gamma) \to L^2(I) \oplus \ell^2(N)$.
		\item Furthermore, $U$ diagonalizes $K \colon H^{1/2}(\Gamma)\to H^{1/2}(\Gamma)$:
		$$U K U^{-1} p(\lambda) = \lambda p(\lambda), \quad U K U^{-1} p(n) = \lambda_n p(n),$$
		for $p \in L^2(I) \oplus \ell^2(N)$, $\lambda \in I$, and $n \in N$.
	\end{itemize}
In particular, the absolutely continuous spectrum of $K \colon H^{1/2}(\Gamma) \to H^{1/2}(\Gamma)$ has multiplicity $1$.
\end{theorem}
\begin{proof}
	For $\lambda \in I$, $\lambda \notin \{\lambda_n\} \cup \{0\}$, let $S_\lambda \colon H^1(\Gamma) \to L^2(\Gamma)$ be defined by
	$$S_\lambda g = (2\pi i)^{-1}(T^{-1}_\lambda - \widetilde{T}_{\lambda}^{-1})g, \quad g \in H^1(\Gamma),$$
	where $\widetilde{T}_{\lambda}$ was defined in the remark after Lemma~\ref{lem:eigenvalues}.
	Then, by \eqref{eq:Tinverse}, $S_\lambda$ maps $H^1(\Gamma)$ into $\ker(K-\lambda)|_{L^2(\Gamma)}$. Thus, by Lemma~\ref{lem:eigspacedim}, there must be $e_\lambda' \in \ker(K-\lambda)|_{L^2(\Gamma)}$ and $e_\lambda'' \in L^2(\Gamma)$ such that
	$$S_\lambda g = \langle g, e_\lambda'' \rangle_{\mathcal{E}'}e_\lambda', \quad g \in H^1(\Gamma).$$
	On the other hand, by Theorem~\ref{thm:spectrum1} the spectral measure $\nu_{g,g}$ has a smooth density at $\lambda$, and therefore
	\begin{align*}
	\langle S_\lambda g, g \rangle_{\mathcal{E'}} &= (2\pi i)^{-1}\lim_{y \to 0^+} \int_{-1}^1 \left( \frac{1}{x-(\lambda + iy)} - \frac{1}{x-(\lambda - iy)} \right) \, d\nu_{g,g}(x) \\
	&= \lim_{y \to 0^+} \frac{1}{\pi}\int_{-1}^1 \frac{y \, d\nu_{g,g}(x)}{(x-\lambda)^2 + y^2} = \nu_{g,g}'(\lambda).
	\end{align*}
	In particular, $\langle g, e_\lambda'' \rangle_{\mathcal{E}'} \langle e_\lambda', g \rangle_{\mathcal{E}'} \geq 0$ for every $g \in H^1(\Gamma)$. Thus there must be a constant $c_\lambda \geq 0$ such that $e''_\lambda = c_\lambda e'_\lambda$. We define
	$$e(\lambda) = \sqrt{c_\lambda} e'_\lambda,$$
	and let $U$ be as in the theorem statement, so that
	$$|Ug(\lambda)|^2 = \langle S_\lambda g, g \rangle = \nu_{g,g}'(\lambda), \quad g \in H^1(\Gamma), \; \lambda \in I \setminus (\{\lambda_n\} \cup \{0\}).$$
	Then, for $g \in H^1(\Gamma)$ and $\lambda \in I$, $\lambda \notin \{\lambda_n\} \cup \{0\}$, we have that
	$$UKg(\lambda) = \langle Kg, e(\lambda) \rangle_{\mathcal{E}'} = \langle g, Ke(\lambda) \rangle_{\mathcal{E}'} = \lambda Ug(\lambda),  $$
	and, similarly, $UKg(n) = \lambda_n Ug(n)$ for $n \in N$.
	
	To complete the proof, we need to show that $U$ extends to a unitary map of $H^{1/2}(\Gamma)$ onto $L^2(I) \oplus \ell^2(N)$. Observe first that for any $g \in H^1(\Gamma)$,
	\begin{equation} \label{eq:Uisometry}
	\int_I |Ug(\lambda)|^2 \, d\lambda + \sum_{n} |Ug(n)|^2 = \int_{-1}^1 \, d\nu_{g,g} = \|g\|^2_{\mathcal{E}'}.
	\end{equation}
	Thus $U$ extends to an isometry $U \colon H^{1/2}(\Gamma) \to L^2(I) \oplus \ell^2(N)$ and, by continuity, the relationship $|Ug(\lambda)|^2 = \nu_{g,g}'(\lambda)$ continues to be valid for every $g \in H^{1/2}(\Gamma)$, for almost every $\lambda \in I$. By polarization, this also implies that
	\begin{equation} \label{eq:polar}
	Ug(\lambda) \overline{Uh(\lambda)} = \nu_{g,h}'(\lambda), \quad \textrm{a.e. } \lambda \in I, \quad g,h \in H^{1/2}(\Gamma).
	\end{equation}
	
	Let $\mathcal{P}$ be the closed linear span of $\{e_n\}_{n \in N}$ in $H^{1/2}(\Gamma)$. By \eqref{eq:polar}, we then have that $Uf(\lambda) = 0$ for $f \in \mathcal{P}$ and almost every $\lambda \in I$. Therefore, in view of \eqref{eq:Uisometry}, $\mathcal{P} \ni f \mapsto \{Uf(n)\}_{n \in N}$ is a unitary map of $\mathcal{P}$ onto $\ell^2(N)$. Note also that $Ug(n) = 0$ for every $g \in \mathcal{P}^\perp$ and $n \in N$.
	
	We now treat the part of $K \colon H^{1/2}(\Gamma) \to H^{1/2}(\Gamma)$ with absolutely continuous spectrum. By Theorem~\ref{thm:spectrum1} and the spectral theorem, there is a measure $d\nu(\lambda) = \nu'(\lambda) \, d\lambda$ on $I$ and a decreasing sequence $B_2 \supset B_3 \supset \cdots$ of Borel sets such that $K \colon \mathcal{P}^\perp \to \mathcal{P}^\perp$ is unitarily equivalent to
	$$M_\nu \oplus M_{\nu|_{B_2}} \oplus M_{\nu|_{B_3}} \oplus \cdots,$$
	where, for a measure $\tau$, $M_{\tau}$ denotes the operator of multiplication by the independent variable $\lambda$ on $L^2(\tau)$. Let $Y \colon \mathcal{P}^\perp \to L^2(\nu) \oplus L^2(\nu|_{B_2}) \oplus \cdots $ be the unitary that realizes this equivalence. The spectral density of two functions $g, h \in \mathcal{P}^\perp$ is then given by
	\begin{equation} \label{eq:specmeasgen}
	\nu'_{g,h} = ((Yg)_1 \overline{(Yh)_1} + (Yg)_2 \overline{(Yh)_2} + \cdots)\nu',
	\end{equation}
	where the $j$th component $(Yg)_j$ is supported in $B_j$ for every $j \geq 2$ and $g \in \mathcal{P}^\perp$. Choose $g \in H^1(\Gamma)$ such that $\nu_{g,g}'$ is non-zero in $(-(1 - \alpha/\pi), 0)$ and in $(0, 1 - \alpha/\pi)$, and let $\tilde{g}$ be the projection of $g$ onto $\mathcal{P}^\perp$. Since $\nu_{g,g}'$ is real analytic in these two intervals by Theorem~\ref{thm:spectrum1}, it can only vanish on a null set in $I$. Noting that $\nu_{g,g}' = \nu'_{\tilde{g},\tilde{g}}$ and applying \eqref{eq:specmeasgen} we therefore see that $\nu$ must be mutually absolutely continuous with the Lebesgue measure $d\lambda$ in $I$. We may therefore assume that $\nu$ is the Lebesgue measure on $I$, $\nu' \equiv 1$. 
	
	Let 
	$$(\mathcal{P}^\perp)_1 = \{g \in \mathcal{P}^\perp \, : \, (Yg)_j = 0 \textrm{ for } j \geq 2\},$$
	so that $(\mathcal{P}^\perp)_1 \ni g \to (Yg)_1 \in L^2(I)$ is unitary. Combining \eqref{eq:polar} and \eqref{eq:specmeasgen} we then have that
	\begin{equation} \label{eq:Uformula}
	Ug(\lambda) \overline{Uh(\lambda)} = (Yg)_1(\lambda) \overline{(Yh)_1(\lambda)}, \quad \textrm{a.e. } \lambda \in I, \quad g,h \in (\mathcal{P}^\perp)_1.
	\end{equation}
	Fix $h \in (\mathcal{P}^\perp)_1$ as the function such that $(Yh)_1 \equiv 1$. Then, by \eqref{eq:Uformula}, $|Uh(\lambda)| = 1$ for almost every $\lambda$, and
	$$Ug(\lambda) = Uh(\lambda) (Yg)_1(\lambda), \quad \textrm{a.e. } \lambda \in I, \quad g \in (\mathcal{P}^\perp)_1.$$
	 We conclude that $g \ni (\mathcal{P}^\perp)_1 \mapsto Ug|_I \in L^2(I)$ is unitary. Since $g \ni \mathcal{P}^\perp \mapsto Ug|_I \in L^2(I)$ is an isometry by \eqref{eq:Uisometry}, this implies that $(\mathcal{P}^\perp)_1 = \mathcal{P}^\perp$, that $K \colon \mathcal{P}^\perp \to \mathcal{P}^\perp$ is unitarily equivalent to $M_{d\lambda} \colon L^2(I) \to L^2(I)$, and that $U$ is unitary.
\end{proof}

\section{Final remarks} \label{sec:remarks}
\subsection{Limiting absorption principle for $K^\ast$} Theorem~\ref{thm:main2} can be restated as a limiting absorption principle for $K^\ast \colon H^{-1/2}(\Gamma) \to H^{-1/2}(\Gamma)$. For simplicity, suppose that $0 < \alpha < \pi$, $\mim \lambda > 0$ and $\lambda \in \widetilde{\Sigma}_{1,\alpha}$. Let $f_\lambda \in H^{-1/2}(\Gamma)$ be the solution of the equation $(K^\ast - \lambda)f_\lambda = g$, for some given $g \in L^2(\Gamma)$. Since $K^\ast = S^{-1}KS$, we then find that
$$f_\lambda = S^{-1}T^{-1}_\lambda Sg$$
Hence, by Theorem~\ref{thm:Tinvert} and Lemma~\ref{lem:eigenvalues}, $\lambda \to f_\lambda$ has an analytic $H^{-1}(\Gamma)$-valued extension to every $\lambda_0 \in (0, 1-\alpha/\pi)$ which is not an eigenvalue of $K^\ast \colon H^{-1/2}(\Gamma) \to H^{-1/2}(\Gamma)$. Note that
\begin{equation} \label{eq:Kastlimabs}
f_{\lambda_0} = S^{-1}T^{-1}_{\lambda_0} S g = c_{\lambda_0}S^{-1} q_{\lambda_0, \alpha} + h_{\lambda_0},
\end{equation}
where $c_{\lambda_0} \in \mathbb{C}$ and $h_{\lambda_0} \in L^2(\Gamma)$. To make this explicit, one should identify the singular behavior of $S^{-1} q_{\lambda, \alpha} \in H^{-1}(\Gamma)$, $\lambda \in \widetilde{\Sigma}_{1,\alpha}$. It seems quite clear that 
\begin{equation} \label{eq:Ssing}
S p_{\lambda, \alpha} = d_\lambda q_{\lambda, \alpha} + o_{\lambda, \alpha},
\end{equation}
 where $0 \neq d_\lambda \in \mathbb{C}$, $o_{\lambda, \alpha} \in H^1(\Gamma)$, and $p_{\lambda, \alpha} \in H^{-1}(\Gamma)$ is the distribution defined by
\begin{multline*}
\langle p_{\lambda, \alpha}, f \rangle_{L^2(\Gamma)} = \int_0^1 \phi(s) s^{-\mu_{\alpha}(\lambda)-1} ((f_- + f_+)(s) - (f_- + f_+)(0)) \, ds \\ 
+ \frac{(f_- + f_+)(0)}{\mu_{\alpha}(\lambda)}\int_0^1 s^{-\mu_{\alpha}(\lambda)} \phi'(s)\, ds, \quad f \in H^1(\Gamma).
\end{multline*}
Note that if $\mim \lambda > 0$, then $p_{\lambda, \alpha} \in L^1(\Gamma)$ and $p_{\lambda, \alpha}(z_\pm(s)) = s^{-\mu_{\alpha}(\lambda)-1}$ for small $s > 0$. For $\mim \lambda \leq 0$, $p_{\lambda, \alpha}$ is obtained by analytic continuation from $\mim \lambda > 0$. Equation \eqref{eq:Kastlimabs} then reads as
$$f_{\lambda_0} = \frac{c_{\lambda_0}}{d_{\lambda_0}} p_{\lambda_0, \alpha} + \tilde{h}_{\lambda_0},$$
where $\tilde{h}_{\lambda_0} \in L^2(\Gamma)$. 

To check the validity of \eqref{eq:Ssing}, one should perform a Mellin analysis of the single layer potential $S$, with similar details to those presented for $K$ in Section~\ref{sec:mellin}. Calculations yielding the analogue of the expansion \eqref{eq:Kexpansion} can be found in \cite{C83, CS83}. It is perhaps easiest to start with $\mim \lambda > 0$ (and $\lambda \in \widetilde{\Sigma}_{1,
\alpha}$), and to then apply analytic continuation in $\lambda$. The specifics are left to the interested reader.

\subsection{Open problems}
As mentioned in the introduction, the analysis we have carried out also applies to piecewise $C^3$-domains in $\mathbb{R}^2$ with finitely many corners. In principle, it should also be possible to extend the analysis to domains in 3D with certain types of conical points, although many aspects will be significantly more involved. Based on the index theory presented in \cite{HP18}, one would then expect to see pieces of absolutely continuous spectrum of arbitrarily large multiplicity. For domains in 3D with edges, parts of the absolutely continuous spectrum should have infinite multiplicity \cite{Per19}.

For 2D domains with corners, an outstanding question is whether $K \colon H^{1/2}(\Gamma) \to H^{1/2}(\Gamma)$ can have a countably infinite number of eigenvalues. Single corner domains featuring eigenvalues embedded in the essential spectrum, that is, in the absolutely continuous spectrum, were constructed in \cite{LS19}. A numerical example exhibiting embedded eigenvalues, obtained by adding three corners to an ellipse, has been presented in \cite{HKL17}.  However, neither of these examples suggest that there can be infinitely many eigenvalues.
\begin{conjecture}
	Let $\Gamma$ be a piecewise $C^3$ Jordan curve with finitely many corners, whose angles lie strictly between $0$ and $\pi$.  Then $K \colon H^{1/2}(\Gamma) \to H^{1/2}(\Gamma)$ has finitely many eigenvalues.
\end{conjecture}
For smooth boundaries $\Gamma$ satisfying certain conditions, it has been demonstrated in the recent preprint \cite{ACL20} that oscillations of the eigenfunctions of the NP operator localize at points of high curvature. A study of this phenomenon for embedded eigenvalues, for domains with singularities, would be very appealing.

Another interesting question is how to recover $\Gamma$ from the spectral data of $K \colon H^{1/2}(\Gamma) \to H^{1/2}(\Gamma)$. For real algebraic $\Gamma$, this question has been successfully considered in \cite{APST19}.

\bibliographystyle{amsplain-nodash} 
\bibliography{NPabscont} 

\providecommand{\bysame}{\leavevmode\hbox to3em{\hrulefill}\thinspace}
\providecommand{\MR}{\relax\ifhmode\unskip\space\fi MR }
\providecommand{\MRhref}[2]{%
  \href{http://www.ams.org/mathscinet-getitem?mr=#1}{#2}
}
\providecommand{\href}[2]{#2}
\begin{thebibliography}{10}

\bibitem{ACL20}
Habib Ammari, Yat~Tin Chow, and Hongyu Liu, \emph{Quantum ergodicity and
  localization of plasmon resonances}, arXiv:2003.03696 [math.AP].

\bibitem{AMRZ17}
Habib Ammari, Pierre Millien, Matias Ruiz, and Hai Zhang, \emph{Mathematical
  analysis of plasmonic nanoparticles: the scalar case}, Arch. Ration. Mech.
  Anal. \textbf{224} (2017), no.~2, 597--658.

\bibitem{APST19}
Habib Ammari, Mihai Putinar, Andries Steenkamp, and Faouzi Triki,
  \emph{Identification of an algebraic domain in two dimensions from a finite
  number of its generalized polarization tensors}, Math. Ann. \textbf{375}
  (2019), no.~3-4, 1337--1354.

\bibitem{AK16}
Kazunori Ando and Hyeonbae Kang, \emph{Analysis of plasmon resonance on smooth
  domains using spectral properties of the {N}eumann-{P}oincar\'{e} operator},
  J. Math. Anal. Appl. \textbf{435} (2016), no.~1, 162--178.

\bibitem{BCC13}
Anne-Sophie Bonnet-Ben~Dhia, Lucas Chesnel, and Xavier Claeys, \emph{Radiation
  condition for a non-smooth interface between a dielectric and a
  metamaterial}, Math. Models Methods Appl. Sci. \textbf{23} (2013), no.~9,
  1629--1662.

\bibitem{BZ19}
Eric Bonnetier and Hai Zhang, \emph{Characterization of the essential spectrum
  of the {N}eumann-{P}oincar\'{e} operator in 2{D} domains with corner via
  {W}eyl sequences}, Rev. Mat. Iberoam. \textbf{35} (2019), no.~3, 925--948.

\bibitem{Carl16}
Torsten Carleman, \emph{{\"U}ber das {N}eumann--{P}oincar\'esche problem f\"ur
  ein gebiet mit ecken}, Almqvist and Wiksels, Uppsala, 1916.

\bibitem{Conway}
John~B. Conway, \emph{Functions of one complex variable. {II}}, Graduate Texts
  in Mathematics, vol. 159, Springer-Verlag, New York, 1995.

\bibitem{C83}
Martin Costabel, \emph{Boundary integral operators on curved polygons}, Ann.
  Mat. Pura Appl. (4) \textbf{133} (1983), 305--326.

\bibitem{CS83}
Martin Costabel and Ernst Stephan, \emph{Curvature terms in the asymptotic
  expansions for solutions of boundary integral equations on curved polygons},
  J. Integral Equations \textbf{5} (1983), no.~4, 353--371.

\bibitem{CS85}
Martin Costabel and Ernst Stephan, \emph{Boundary integral equations for mixed
  boundary value problems in polygonal domains and {G}alerkin approximation},
  Mathematical models and methods in mechanics, Banach Center Publ., vol.~15,
  PWN, Warsaw, 1985, pp.~175--251.

\bibitem{CO01}
R.~V. Craster and Yu.~V. Obnosov, \emph{Four-phase checkerboard composites},
  SIAM J. Appl. Math. \textbf{61} (2001), no.~6, 1839--1856.

\bibitem{GGK90}
Israel Gohberg, Seymour Goldberg, and Marinus~A. Kaashoek, \emph{Classes of
  linear operators. {V}ol. {I}}, Operator Theory: Advances and Applications,
  vol.~49, Birkh\"{a}user Verlag, Basel, 1990.

\bibitem{HelTut}
Johan Helsing, \emph{Solving integral equations on piecewise smooth boundaries
  using the {RCIP} method: a tutorial}, arXiv:1207.6737v9 [physics.comp-ph].

\bibitem{HKL17}
Johan Helsing, Hyeonbae Kang, and Mikyoung Lim, \emph{Classification of spectra
  of the {N}eumann-{P}oincar\'{e} operator on planar domains with corners by
  resonance}, Ann. Inst. H. Poincar\'{e} Anal. Non Lin\'{e}aire \textbf{34}
  (2017), no.~4, 991--1011.

\bibitem{HP13}
Johan Helsing and Karl-Mikael Perfekt, \emph{On the polarizability and
  capacitance of the cube}, Appl. Comput. Harmon. Anal. \textbf{34} (2013),
  no.~3, 445--468.

\bibitem{HP18}
Johan Helsing and Karl-Mikael Perfekt, \emph{The spectra of harmonic layer
  potential operators on domains with rotationally symmetric conical points},
  J. Math. Pures Appl. (9) \textbf{118} (2018), 235--287.

\bibitem{HR19}
Johan Helsing and Andreas Ros\'en, \emph{Dirac integral equations for
  dielectric and plasmonic scattering}, arXiv:1911.00788 [math.AP].

\bibitem{HLY17}
Hyeonbae Kang, Mikyoung Lim, and Sanghyeon Yu, \emph{Spectral resolution of the
  {N}eumann-{P}oincar\'{e} operator on intersecting disks and analysis of
  plasmon resonance}, Arch. Ration. Mech. Anal. \textbf{226} (2017), no.~1,
  83--115.

\bibitem{KPS07}
Dmitry Khavinson, Mihai Putinar, and Harold~S. Shapiro, \emph{Poincar\'{e}'s
  variational problem in potential theory}, Arch. Ration. Mech. Anal.
  \textbf{185} (2007), no.~1, 143--184.

\bibitem{KMP00}
Natan Krugljak, Lech Maligranda, and Lars~Erik Persson, \emph{On an elementary
  approach to the fractional {H}ardy inequality}, Proc. Amer. Math. Soc.
  \textbf{128} (2000), no.~3, 727--734.

\bibitem{LS19}
Wei Li and Stephen~P. Shipman, \emph{Embedded eigenvalues for the
  {N}eumann--{P}oincar\'e operator}, J. Integral Equations Applications (2019),
  to appear.

\bibitem{LM72}
J.-L. Lions and E.~Magenes, \emph{Non-homogeneous boundary value problems and
  applications. {V}ol. {I}}, Springer-Verlag, New York-Heidelberg, 1972,
  Translated from the French by P. Kenneth, Die Grundlehren der mathematischen
  Wissenschaften, Band 181.

\bibitem{Milt01}
Graeme~W. Milton, \emph{Proof of a conjecture on the conductivity of
  checkerboards}, Journal of Mathematical Physics \textbf{42} (2001), no.~10,
  4873--4882.

\bibitem{Mit02}
Irina Mitrea, \emph{On the spectra of elastostatic and hydrostatic layer
  potentials on curvilinear polygons}, J. Fourier Anal. Appl. \textbf{8}
  (2002), no.~5, 443--487.

\bibitem{Per19}
Karl-Mikael Perfekt, \emph{The transmission problem on a three-dimensional
  wedge}, Arch. Ration. Mech. Anal. \textbf{231} (2019), no.~3, 1745--1780.

\bibitem{PP14}
Karl-Mikael Perfekt and Mihai Putinar, \emph{Spectral bounds for the
  {N}eumann-{P}oincar\'{e} operator on planar domains with corners}, J. Anal.
  Math. \textbf{124} (2014), 39--57.

\bibitem{PP17}
Karl-Mikael Perfekt and Mihai Putinar, \emph{The essential spectrum of the
  {N}eumann-{P}oincar\'{e} operator on a domain with corners}, Arch. Ration.
  Mech. Anal. \textbf{223} (2017), no.~2, 1019--1033.

\bibitem{Shele90}
V.~Yu. Shelepov, \emph{On the index and spectrum of integral operators of
  potential type along {R}adon curves}, Mat. Sb. \textbf{181} (1990), no.~6,
  751--778.

\bibitem{Verchota84}
Gregory Verchota, \emph{Layer potentials and regularity for the {D}irichlet
  problem for {L}aplace's equation in {L}ipschitz domains}, J. Funct. Anal.
  \textbf{59} (1984), no.~3, 572--611.

\bibitem{Wendland09}
W.~L. Wendland, \emph{On the double layer potential}, Analysis, partial
  differential equations and applications, Oper. Theory Adv. Appl., vol. 193,
  Birkh\"{a}user Verlag, Basel, 2009, pp.~319--334.

\end{thebibliography}
\end{document}